\documentclass[11pt]{amsproc}

\usepackage{a4wide}
\usepackage[pagebackref=true,colorlinks,linkcolor=red,citecolor=blue,urlcolor=blue,hypertexnames=true]{hyperref}
\usepackage{setspace}
\usepackage{amsmath}
\usepackage{amssymb}
\usepackage{amsfonts}
\usepackage{amsthm}
\usepackage{mathtools}
\usepackage{array}

\theoremstyle{plain}
\newtheorem{thm}{Theorem}[section]
\theoremstyle{definition}
\newtheorem{defn}[thm]{Definition}
\newtheorem{lem}[thm]{Lemma}
\newtheorem{prop}[thm]{Proposition}
\newtheorem{cor}[thm]{Corollary}
\theoremstyle{remark}
\newtheorem{rem}{Remark}[section]

\newtheorem*{expls}{Examples}

\newcommand{\tb}[1]{\textbf{#1}}

\newcommand{\tn}[1]{\textnormal{#1}}
\newcommand{\til}[1]{\widetilde{#1}}
\newcommand{\norm}[1]{\left\Vert#1\right\Vert}
\newcommand{\set}[1]{\left\{#1\right\}}

\newcommand{\abs}[1]{\left\lvert#1\right\rvert}

\newcommand{\eps}{\varepsilon}
\newcommand{\bbC}{\mathbb{C}}

\newcommand{\bbN}{\mathbb{N}}

\newcommand{\bbR}{\mathbb{R}}

\newcommand{\cC}{\mathcal{C}}

\newcommand{\cH}{\mathcal{H}}

\newcommand{\cM}{\mathcal{M}}
\newcommand{\cN}{\mathcal{N}}

\newcommand{\cP}{\mathrm{P}}

\newcommand{\cU}{\mathrm{U}}
\newcommand{\cZ}{\mathcal{Z}}
\newcommand{\cSU}{\mathrm{S}\mathrm{U}}

\DeclareMathOperator{\complex}{\mathrm{i}}

\DeclareMathOperator{\diam}{diam}
\DeclareMathOperator{\diag}{diag}

\DeclareMathOperator{\Proj}{Proj}

\DeclareMathOperator{\tr}{tr}

\title[Bounded Normal Generation]{Bounded Normal Generation and Invariant \\ Automatic Continuity}

\author{Philip A. Dowerk}
\address{P.A.D., Analysis Section, KU Leuven, 3001 Leuven, Belgium}
\email{philip.dowerk@wis.kuleuven.be}

\author{Andreas Thom}
\address{A.T., Institut f\"ur Geometrie, TU Dresden, 01062 Dresden, Germany }
\email{andreas.thom@tu-dresden.de}

\begin{document}

\onehalfspace

\begin{abstract}
We study the question how quickly products of a fixed conjugacy class in the projective unitary group of a II${}_1$-factor von Neumann algebra cover the entire group. Our result is that the number of factors that are needed is essentially as small as permitted by the $1$-norm -- in analogy to a result of Liebeck-Shalev for non-abelian finite simple groups. As an application of the techniques, we prove that every homomorphism from the projective unitary group of a II${}_1$-factor to a polish SIN group is continuous. Moreover, we show that the projective unitary group of a II${}_1$-factor carries a unique polish group topology.
\end{abstract}

\maketitle

\tableofcontents

\section{Introduction}

It is a fundamental question in group theory to ask under which conditions one element of a group $G$ is the product of conjugates of another element in $G$. If for every $g\in G,\ g\neq 1,$ its conjugacy class and that of its inverse generate $G$ in finitely many steps we say that $G$ has the bounded normal generation property, or property {\rm (BNG)}, a strong form of simplicity -- see Definition \ref{def_BNG}. Our study is focussed on projective unitary groups of finite factorial von Neumann algebras. For $G:= {\rm PU}(\cM)$ with $\cM$ a finite factor, we find an explicit normal generation function, i.e., an integer-valued function on $G \setminus \{1\}$ which for every $g\in G$ gives a bound on the number of steps to generate the whole group with the conjugacy class of $g$ and $g^{-1}$. Finite factors a classified into types I$_{n}$ for $n \in \mathbb N$ and type II$_1$, where the II$_1$ case contains a variety of interesting different von Neumann algebras, that are subject of intensive investigation for over almost 70 years now. However, our results are even new and interesting in the type I$_n$ situation, where the corresponding projective unitary group is just the compact Lie group ${\rm PU}(n)$.

Indeed, for compact metrizable simple groups it is not hard to obtain property {\rm (BNG)} qualitatively (i.e., without an explicit normal generation function) via a Baire category argument, cf.\ Proposition \ref{prop_compact_BNG}. 
However, it is hard to find explicit and sharp normal generation functions even in the case of finite simple groups. Liebeck and Shalev provided a minimal normal generation function for non-abelian finite simple groups $G$ in their seminal work \cite{LS-01} and used this result to obtain many interesting applications. Their normal generation function is of the form $f(g)=c\log(\abs{G})/ \log(\abs{g^G})$, where $c$ is a universal constant and $g^G$ denotes the conjugacy class of $g\in G$. (It is easy to see that already for reasons of cardinality, this must be optimal up to a multiplicative constant.)
In 2012, Nikolov and Segal proved property (BNG) for compact connected simple Lie groups \cite[Proposition 5.11]{NS-12}. They also provided a normal generation function, which however was depending strongly on the rank of the group. We refine their results to get Theorem \ref{thm_sun}, which provides a rank-independent normal generation function. This is our first main result -- it corrects a mistake in the work by Stolz and the second author \cite[Lemma 4.15]{ST-13}, where similar results were claimed.

In general, having a dimension-independent  result suggests the existence of an infinite-dimensional analogue. 
Indeed, using our result for $\cP\cU(n)$ we can prove property (BNG) for the projective unitary group $\cP\cU(\cM)$ of a $\mathrm{II_1}$-factor $\cM$, endowed with the strong operator topology (see Section \ref{sec_II_1}). This is one of our main results. Projective unitary groups of II$_1$-factors have strong similarities with compact simple groups and have been the subject of intensive study ever since  von Neumann introduced them in his groundbreaking work, see \cite{MvN-36,MvN-37,MvN-43}. 

The proofs of these results are rather technical and use finite-dimensional approximation at various steps to reduce to the case of the projective unitary group $\cP\cU(n)$, $n\in\bbN$. Our approach provides a new proof of the algebraic simplicity of projective unitary groups of $\mathrm{II_1}$-factors which was discovered by de la Harpe in \cite{dlH-76} and a quantitative version of Broise's results in \cite{Bro-67}, stating that every unitary in a $\mathrm{II_1}$-factor is the product of $32$ conjugates of any symmetry with trace $0$. Topological simplicity of projective unitary groups of $\mathrm{II_1}$-factors, endowed with the uniform topology, was proved earlier by Kadison in \cite{Kad-52}.

Let us now state the main theorems on bounded normal generation more explicitly. For $t\geq 0$ and $u$ an element of the unitary group $\cU(\cM)$ of a finite factorial von Neumann algebra we define
$$\ell_t(u)\coloneqq \inf_{\lambda\in{S^1}}\mu_t(1-\lambda u),$$
where $t \mapsto \mu_t(x)$ denotes the generalized singular values of the operator $x \in \cM$, as studied by Fack-Kosaki in \cite{FK-86}. For definitions see Section \ref{sec_proj_s_numbers}. For factors of type I${}_n$, the domain of $\ell_t$ and $\mu_t$ is $\{0,\dots,n-1\}$, whereas it is $[0,1]$ for factors of type II${}_1$.

It is easy to see that if $u$ is a product of $k$ conjugates of $v$, then $\ell_{kt}(u)\leq k\ell_t(v)$ for all $t\geq 0$, see  Proposition \ref{prop_easy_dir} in Section \ref{sec_proj_s_numbers}. Our aim will be to prove some weak converse to this observation. We first study the case $\cP\cU(n)$ and present a result that corrects a mistake in the proof of \cite[Lemma 4.15]{ST-13}.

\begin{thm}\label{thm_main_5}
 Let $G:=\cP\cU(n)$ be the projective unitary group, where $n\in\bbN,\ n\geq 2$. Let $u,v\in G$ and $m\in\bbN$. If $\ell_{0}(u)\leq m\ell_t(v)$ for all $t= 0,1,
 \ldots,s-1$ then 
 $$u\in (v^G\cup v^{-G})^{24m\lceil n/s\rceil}.$$
In particular, $G$ has property ${\rm (BNG)}$.
\end{thm}

The most interesting case from our viewpoint is the $\mathrm{II_1}$-factor case. The study of this case and the proof of the corresponding result is spread over Sections \ref{sec_II_1} and \ref{sec_alg_II_1}, using a careful and non-trivial reduction to the matrix case.

\begin{thm}\label{thm_main_1}
 Let $G$ denote the projective unitary group of a separable $\mathrm{II_1}$-factor. Let $u,v\in G$ and $m\in\bbN$. Assume that $u$ has finite spectrum and rational weights. If $\ell_{0}(u)\leq m\ell_t(v)$ for all $t\in[0,s]$, then 
 $$u\in (v^G\cup v^{-G})^{cm\lceil 1/s\rceil}$$
 for some universal constant $c \in \bbN$.
\end{thm}
In combination with results of Broise, this implies that $G$ has
property {\rm (BNG)}.
%
We now present a formulation of Theorem \ref{thm_main_1} with a suitable normal generation function. For $x\in\cM$ with $\cM$ a  finite factor, we define 
$$\ell(x)\coloneqq \inf_{\lambda \in S^1} \|1-\lambda x\|_1.$$
\begin{thm} \label{normalgenfun}
 Let $G$ denote the projective unitary group of a separable $\mathrm{II_1}$-factor. For some universal constant $c>0$ the function $f:G\setminus\set{1}\rightarrow \bbN$, given by
 $$f(v)\coloneqq  c \cdot |\log \ell(v)| \cdot \ell(v)^{-1},$$
is a normal generation function for $G$. That is, 
$G=( v^G\cup v^{-G} )^{k}$ for every $k\geq f(v),\ v\in G\setminus\set{1}$.
\end{thm}
It is easy to see that any normal generation function must satisfy $f(v) \geq 2 \ell(v)^{-1}$ since the diameter of $G$ is equal to $2$, so that we are not far from an optimal answer.

The techniques developed in the study of II$_1$-factor case lead to our second main result, which is on invariant automatic continuity for the groups $\cP\cU(n)$, $\cSU(n)$ and $\cP\cU(\cM)$, where $\cM$ is a separable $\mathrm{II_1}$-factor, see Section \ref{cha_autocont}. Recall, that a polish group is called SIN if the topology has a basis of conjugation invariant neighborhoods at the neutral element.

\begin{thm} For every finite factor $\cM$, every homomorphism from $\cP\cU(\cM)$ to a polish {\rm SIN} group must be continuous. Moreover, the group $\cP\cU(\cM)$ carries a unique polish group topology.
\end{thm}

The most important preliminaries are covered in the first sections, namely we define property (BNG), study basic properties of length functions, and recall the definition and some important properties of the generalized $s$-numbers for semi-finite von Neumann algebras based on the article of \cite{FK-86} of Fack and Kosaki. We will then define generalized projective $s$-numbers in Section \ref{sec_proj_s_numbers} in this context and prove some properties that are required in the preceding sections. 

\section{Bounded normal generation and length functions}\label{sec_BNG}
In this section we define the main concept studied in this article, the so-called bounded normal generation property for groups. Let $G$ be a group. We denote by $g^G$ the conjugacy class of $g\in G$ and by $g^{-G}$ the conjugacy class of $g^{-1}$.

\begin{defn}\label{def_BNG} 
(i) Let $g$ be an element of a group $G$. If there exists $k\in\bbN$ such that $G=(g^G\cup g^{-G})^k$
then we call $g$ a \textbf{uniform normal generator for $G$}\index{uniform normal generator}. If we want to emphasize the number $k$ we will write that $g$ is a $k$-uniform normal generator \index{$k$-uniform normal generator}. 

(ii) A group $G$ has the \textbf{bounded normal generation property}\index{bounded normal generation property} or \textbf{property {\rm (BNG)}}\index{property {\rm (BNG)}} if every non-trivial element is a uniform normal generator. That is, there exists a function $f:G\setminus\set{1}\rightarrow\bbR$ such that
$G=(g^G\cup g^{-G})^{k}$ for every $g\neq 1$ and $k \geq f(g)$. We call $f$ a \textbf{normal generation function}\index{normal generation function}. 

\end{defn}



It is clear that every group with property ${\rm (BNG)}$ is simple -- an example for a group that is simple but does not have property {\rm (BNG)} is the infinite alternating group of all finitely supported even permutations on $\bbN$.
In the case of compact simple groups one can easily show property {\rm (BNG)} via a Baire category argument, see Proposition \ref{prop_compact_BNG}. However, getting a concrete normal generation function is often much harder even in the case of non-abelian finite simple groups. 
Let us list some known examples.

\begin{expls}
(i) Every non-abelian finite simple group $G$ has the bounded normal generation property with normal generation function given by 
  $$f(g)\coloneqq c \cdot \log\abs{G}/\log\abs{g^G}  \quad \mbox{ for }g\in G\setminus\set{1},$$ 
see \cite[Theorem 1]{LS-01}. One can even omit $g^{-G}$ in the generation process.

(ii) Compact connected simple Lie groups have property {\rm (BNG)}, see \cite[Proposition 5.11]{NS-12}. Nikolov and Segal also provide a normal generation function, but its dependence on the rank is far from being optimal. We will provide an improved normal generation function for $\cP\cU(n)$ via a study of projective singular values and reprove property {\rm (BNG)} for $\cP\cU(n)$ in Section \ref{sec_I_n}. In Corollary \ref{cor_bdd_rk} we provide the following dimension-dependent normal generation function for $\cP\cU(n)$: 
  $$f(u)\coloneqq 16n/\inf_{\lambda\in{S^1}}\norm{1-\lambda u}\quad \mbox{ for } u\in\cP\cU(n)\setminus\set{1}.$$
  
(iii) The connected component $\cP\cU_1(\cC)$ of the identity of the projective unitary group of the Calkin algebra $\cC$ has property {\rm (BNG)}, see \cite{DT}. A normal generation function is given by 
  $$f(u)\coloneqq 40/ \inf_{\lambda\in{S^1}}\norm{1-\lambda u}_{\mathrm{ess}} \quad \mbox{ for } u\in \cP\cU_1(\cC)\setminus\set{1}.$$
\end{expls}

As mentioned above, using a Baire category argument one can show the following qualitative result for compact simple groups. Let us present the proof for convenience.

\begin{prop}\label{prop_compact_BNG}
Every compact topologically simple group $G$ has property {\rm (BNG)}.
\end{prop}
\begin{proof}
Topological simplicity implies algebraic simplicity for compact groups by \cite[Theorem 9.90]{HM-06}. 
Observe that for any $g\in G\setminus{\{1\}}$ the set $ \bigcup_{n\in\bbN}(g^G\cup g^{-G})^n$ forms a non-trivial normal subgroup of $G$ and note that $g^G$ is compact as the continuous image of the compact set $G$ under conjugation. Since $G$ is simple we have 
$G=\bigcup_{n\in\bbN}(g^G \cup g^{-G})^n.$
 For $k\in\bbN$ we define $C^k\coloneqq \bigcup_{n\leq k}(g^G\cup g^{-G})^n$. Since $G$ is compact and the sets $C^k$ are closed and we can apply the Baire category theorem to obtain the existence of $m\in\bbN$ such that
$\mathrm{int}(C^m)\neq\emptyset.$
 Assume that $U\subseteq C^m$ is non-empty and open and let $V\coloneqq UU^{-1}\subseteq C^{2m}$. Since $1\in V$ we have $C^m\subseteq VC^m$. Thus $\bigcup_{n\in\bbN}VC^n$ is an open covering of $G$. Now, compactness of $G$ implies that there exists $m'\in\bbN$ such that 
 $$G=\bigcup_{n\leq m'}VC^n\subseteq\bigcup_{n\leq m'}C^{n+2m} = \bigcup_{n \leq m'+2m }(g^G \cup g^{-G})^n.$$
Since $g \neq 1$ was arbitrary, $G$ has property {\rm (BNG)}.
\end{proof}



Natural candidates for normal generation functions are closely related to so called invariant length functions, studied in this context already by Stolz and the second author in \cite{ST-13}.

\begin{defn}\label{def_length_function}
Let $G$ be a group. We say that a function 
$\ell: G\rightarrow[0,\infty)$
is a \textbf{pseudo length function}\index{pseudo length function} on $G$ if for all $g,h\in G$ the following properties hold:
\begin{enumerate}
 \item[(i)]  $\ell(1)=0$;
 \item[(ii)] $\ell(g)=\ell(g^{-1})$;
 \item[(iii)] $\ell(gh)\leq\ell(g)+\ell(h)$.
\end{enumerate}
If $\ell$ is a pseudo length function which additionally satisfies that $\ell(g)=0$ implies $g=1$, then $\ell$ is called \textbf{length function}\index{length function}. A pseudo length function $\ell$ is called \textbf{invariant}\index{pseudo length function!invariant} if one has
$\ell(hgh^{-1})=\ell(g)$ for all $g,h\in G.$
\end{defn}

If $\ell$ be a length function on a group $G$, then
the function $d(g,h)\coloneqq \ell(gh^{-1})$ defines a  metric on $G$. Conversely a metric $d$ on $G$ induces a  length function on $G$ by $\ell(g)\coloneqq d(1,g),\ g\in G$ and $\ell$ is invariant if and only if $d$ is bi-invariant, i.e., $d(gh,gk)=d(h,k)=d(hg,kg)$ for all $g,h,k \in G$. For any subset $X \subset G$, we set 
$X_{\varepsilon} := \{g \in G \mid \exists x \in X : d(g,x)< \varepsilon \}.$ We define the diameter as usual $\diam_{\ell}(G)=\sup_{h \in G} \ell(h)$.

Invariant length functions can be used to provide
lower bounds for normal generation functions.

\begin{prop}\label{prop_lower_bound}
 Let $G$ be a group with property {\rm (BNG)} and normal generation function $f$. Assume that $\ell$ is an invariant length function on $G$. Then,
$f(g)\geq {\rm diam}_{\ell}(G) \ell(g)^{-1}$ for all $g\in G\setminus\set{1}.$
\end{prop}
\begin{proof}
 Let $g\in G\setminus\set{1}$ and assume that $h\in (g^G\cup g^{-G})^{f(g)}$.
 Since $\ell$ is an invariant length function, we have
$\ell(h)\leq f(g)\ell(g).$
The claim follows since $\diam_{\ell}(G)=\sup_{h \in G} \ell(h)$ by definition.
\end{proof}

We present some examples of length functions.

\begin{expls}
(i) Let $G$ be a finite simple group. Then the conjugacy length $$\ell_{\mathrm{conj}}(g)\coloneqq\frac{\log\abs{g^G}}{\log\abs{G}}$$ defines an invariant length function. In fact, Liebeck and Shalev \cite{LS-01} showed that $ 1/\ell_{\mathrm{conj}}(\cdot) $ is (up to a multiplicative constant) also a normal generation function for finite simple groups -- and hence, as already mentioned, their result is optimal up to a multiplicative constant.

(ii) Let $\mathcal{C}$ denote the Calkin algebra on the separable infinite-dimensional Hilbert space $\mathcal{H}$. Write $\cP\cU_1(\mathcal{C})$ for the connected component of the neutral element in the projective unitary group of $\mathcal{C}$. The essential norm $\norm{.}_{\mathrm{ess}}$ on $\mathcal{C}$ induces a length function on $\cP\cU_1(\cC)$ via
  $$\ell_{\mathrm{ess}}(u)\coloneqq \inf_{\lambda\in{S^1}}\norm{1-\lambda u}_{\mathrm{ess}}, u\in\cP\cU_1(\cC).$$
  In \cite{DT} we show that $40 / \ell_{\mathrm{ess}}(\cdot)$ defines a normal generation function -- again optimal up to a multiplicative constant. 

(iii) Let $\cM$ be a $\mathrm{II_1}$-factor. The norms $\norm{\cdot},\norm{\cdot}_1$ and $\norm{\cdot}_2$ induce invariant length functions on $\cU(\cM)$. 
It follows that 
$$\ell(u)\coloneqq\inf_{\lambda\in{S^1}}\norm{1-\lambda u}_1,\ u\in\cU(\cM),$$
defines an invariant length function on $\cP\cU(\cM)$. Our main result on normal generation functions does not quite reach $c/ \ell(\cdot)$, but we prove in Theorem \ref{normalgenfun} that $u \mapsto c \cdot |\log \ell(u)| / \ell(u)$ is a normal generation function.
\end{expls}

Let us conclude this section by proving some elementary lemma that we need later in the proofs, namely those which rely on finite-dimensional approximation. We show that products of $\eps$-thickened conjugacy classes of topological groups with compatible bi-invariant metric behave well under $\eps$-thickening.

\begin{lem}\label{lem_epsconj}
 Let $G$ be a topological group equipped with a compatible bi-invariant metric $d$ and let $\eps>0$. Then, $(((g^G)_{\eps})^n)_{\eps}^{}\subseteq ((g^G)^n)_{(n+1)\eps}$ for all $n\in\bbN$.
\end{lem}
\begin{proof}
 Let $h\in(((g^G)_{\eps})^n)_{\eps}^{}$ and assume that $g_{i,\eps}$ for $i=1,\ldots,n,$ are elements of $(g^G)_{\eps}$ satisfying $d(h,g_{1,\eps}\cdots g_{n,\eps})<\eps$. Then there are elements $g_1,\ldots, g_n\in g^G$ such that $d(g_i,g_{i,\eps})<\eps$. Using the bi-invariance of $d$, we obtain
 \begin{align*}
  d(h,g_1\cdots g_n)&\leq d(h,g_{1,\eps}\cdots g_{n,\eps})+d(g_{1,\eps}\cdots g_{n,\eps},g_1g_{2,\eps}\cdots g_{n,\eps})\\
  &\quad +\ldots+d(g_1\cdots g_{n-1}g_{n,\eps},g_1\cdots g_n)\\
  &< \eps+d(g_{1,\eps},g_1)+\ldots+d(g_{n,\eps},g_n)\\
  &< (n+1)\eps,
 \end{align*}
 which shows that $h\in ((g^G)^n)_{(n+1)\eps}$.
\end{proof}

\section{Products of symmetries}\label{sec_Broise}

In this section we provide a detailed analysis and improvement of \cite[Theorem 1]{Bro-67} by Broise. The original version states that for every unitary element in a $\mathrm{II_1}$-factor there exists $n \in \mathbb N$ such that $u$ can be written as $u=v_1\cdot\ldots\cdot v_n$, where $v_i=s_ir_i s_i r_i$ and $r_i,s_i$ are symmetries. From the original formulation of the result it is not clear whether $n$ depends on $u$ and if $s_i,r_i$ or $v_i$ can be chosen close to the identity if $u$ is close to the identity.
Our main result in this section is the following theorem -- in this improved form it will be crucial in the proof of our main results.

\begin{thm} \label{broise_improved}
Let $\cM$ be a $\mathrm{II_1}$-factor. Every $u\in\cU(\cM)$ can be decomposed into factors $u=u_1\cdot\ldots\cdot u_8$ with $u_i\in\cU(\cM)$, $1 \leq i \leq 8$, such that for each $u_i$ there is a projection $p_i\in\Proj(\cM),\ \tau(p_i)=1/3$, such that under an isomorphism of $\cM$ to $p_i\cM p_i\otimes M_{3\times 3}(\bbC)$ the element $u_i$ has the form 
  \begin{align*}
u_i=   \begin{pmatrix}
    1 & 0 & 0\\
    0 & w_i & 0\\
    0 & 0 & w_{i}^*
   \end{pmatrix}
  \end{align*}
  for some $w_{i}\in \cU(p_i\cM p_i)$.
Moreover, for all $\eps>0$ there exists $\delta>0$ such that if $\norm{1-u}<\delta$, then 
$\norm{1-u_i}_2< \eps.$
\end{thm}

This easily implies Broise's original result since
 \begin{align*}
  \begin{pmatrix} w &0\\0 &w^*\end{pmatrix}=\begin{pmatrix} u^2 &0\\0 &u^{*2}\end{pmatrix}
  =\begin{pmatrix} 0 &u\\u^* &0\end{pmatrix}\begin{pmatrix} 0 &1\\1 &0\end{pmatrix}\begin{pmatrix} 0 &u\\u^* &0\end{pmatrix}\begin{pmatrix} 0 &1\\1 &0\end{pmatrix},
 \end{align*}
for any $u$ with $u^2=w$. Indeed, it follows that any $u \in \cU(\cM)$ is the product of 32 symmetries of trace zero. Since all those symmetries are conjugate, we obtain:

\begin{cor} \label{broisecor} 
Let $\cM$ be a II$_1$-factor, $G: = \cU(\cM)$, and $s \in \cU(\cM)$ be a symmetry with $\tau(s)=0$. Then, we have $G= (s^{G})^{32}.$ 
\end{cor}
This is the prototype of a bounded normal generation result in the type II$_1$ situation, and we will use it in the final step of the generation process.

\subsection{The work of Broise}\label{subsec_prod_sym}

The following result is a slightly strengthened form of \cite[Lemma 5]{Bro-67}. For the proof of the main result in this section we will only need one case of the following lemma, but the second case will become important later.
\begin{lem}[]\label{lem_bro2}
 Assume that $\cM$ is a $\mathrm{II_1}$-factor and $p\in\Proj(\cM)$. Let $n\in\set{2,3}$. Suppose that $\set{w_{i,j}}_{1\leq i,j\leq n}$ and $\set{w_i}_{1\leq i\leq n}$ are families of elements in $\cM$ satisfying the following three conditions:
 \begin{enumerate}
  \item[(1)] $w_{i,l}w_{l,j}=w_{i,j}$ and $(w_{i,j})^*=w_{i,j}^*=w_{j,i}$ for all $1\leq i,j\leq n$.
  \item[(2)] $p$ and $\set{w_{i,i}}_{1\leq i\leq n}$ are pairwise orthogonal projections.
  \item[(3)] $p+\sum_{i=1}^n w_i\in\cU(\cM)$ and $w_iw_{i,i}=w_{i,i}w_i=w_i$ for all $1\leq i \leq n$.
 \end{enumerate}

In case $n=2$ and $w_2=w_{2,1}w_1^*w_{1,2}$, then $p+w_1+w_2=stst$ for some symmetries $s,t\in\cU(\cM)$ satisfying $\tau(s)=\tau(t)=0$.
In case $n=3$ and $w_3=w_{3,2}w_2^*w_{2,1}w_1^*w_{1,3}$, then $p+w_1+w_2+w_3=s_1t_1s_1t_1\cdot s_2t_2s_2t_2$ for some symmetries $s_1,s_2,t_1,t_2\in\cU(\cM)$ satisfying $\tau(s_i)=\tau(t_i)=0,\ i=1,2$.
\end{lem}
\begin{proof}
 Consider first the case $p=0$. Put $\cM_{1,1}\coloneqq w_{1,1}\cM w_{1,1}$ and $\cN\coloneqq\cM_{1,1}\otimes M_{n\times n}(\bbC)$. Then $\psi:\cM\rightarrow\cN,\ x\mapsto (x_{i,j})=(w_{1,i}xw_{j,1})$ is a homomorphism from $\cM$ to the matrix algebra $\cN$. Conditions (1),(2) and (3) imply that $\psi$ is an isomorphism. By a corollary to Proposition I.2 in \cite{Dix-81}, $\cM_{1,1}$ is again a $\mathrm{II_1}$-factor, hence $\cN$ is a $\mathrm{II_1}$-factor by \cite[Proposition 11.2.20]{KR}. 
 
In case $n=2$, the assumptions imply that
 \begin{align*}
  \psi(w_1+w_2)&=\begin{pmatrix} w_{1,1}w_1w_{1,1} &w_{1,1}w_1w_{2,1}\\w_{1,2}w_1w_{1,1} &w_{1,2}w_1w_{2,1}\end{pmatrix}+\begin{pmatrix} w_{1,1}w_2w_{1,1} &w_{1,1}w_2w_{2,1}\\w_{1,2}w_2w_{1,1} &w_{1,2}w_2w_{2,1}\end{pmatrix}\\
&=\begin{pmatrix} w_{1,1}w_1w_{1,1} &0\\0 &w_{1,2}w_{2,1}w_1^*w_{1,2}w_{2,1}\end{pmatrix}
=\begin{pmatrix} w_1 &0\\0 &w_{1,1}w_1^*w_{1,1}\end{pmatrix}
=\begin{pmatrix} w_1 &0\\0 &w_1^*\end{pmatrix}.
 \end{align*}
 By condition (3), $w_1+w_2$ is unitary in $\cM$, hence $w_1=w_{1,1}(w_1+w_2)w_{1,1}$ is unitary in $\cM_{1,1}$.   

In case $n=3$, put $\til{w}_2\coloneqq w_{1,2}w_2w_{2,1}$. The condition $w_3=w_{3,2}w_2^*w_{2,1}w_1^*w_{1,3}$ implies that 
 \begin{align*}
  \psi(w_1+w_2+w_3)
  &=\begin{pmatrix}w_1&0&0\\0&w_{1,2}w_2w_{2,1}&0\\0&0&w_{1,3}w_{3,2}w_2^*w_{2,1}w_1^*w_{1,3}w_{3,1}    \end{pmatrix}\\
  &=\begin{pmatrix}w_1&0&0\\0&\til{w}_2&0\\0&0&\til{w}_2^*w_1^*    \end{pmatrix} =\begin{pmatrix}1&0&0\\0&\til{w}_2&0\\0&0&\til{w}_2^*    \end{pmatrix}\begin{pmatrix}w_1&0&0\\0&1&0\\0&0&w_1^*    \end{pmatrix}.
 \end{align*}
Note that $w_1$ and $\til{w}_2$ are unitary in $\cM_{1,1}$.   
Now consider the case $p\neq 0$, i.e., $\tau(p)>0$. Decompose $p$ into $\sum_{i=1}^np_i$, where $p_i$ are equivalent orthogonal projections. Then define $\til{w}_i\coloneqq w_i+p_i$, $i=1,\ldots,n$. Hence $\sum_{i=1}^n\til{w}_i\in\cU(\cM)$ by condition (3). Adjust the system $\{w_{i,j}\}$ by setting $\til{w}_{i,j}\coloneqq w_{i,j}+x_i^*x_j$, where $x_l$ are the partial isometries such that $x_l^*x_l=p_l$ and $x_lx_l^*=p_{1}$. The families $\set{\til{w}_{i,j}}$ and $\set{\til{w}_i}$ clearly satisfy conditions (2) and (3). We check condition (1). We have $(\til{w}_{i,j})^*=w_{j,i}+x_j^*x_i=\til{w}_{j,i}$ and 
  \begin{align*}
  \til{w}_{i,l}\til{w}_{l,j}&=w_{i,j}+x_i^*x_lx_l^*x_j=w_{i,j}+x_i^*x_ix_i^*x_j  =\til{w}_{i,j}.
  \end{align*}
  That is, we may use the first part of the proof on the adjusted families. Finally, let us show that if the assumptions in case $n=2$ are satisfied for $w_2$, then we have 
  \begin{align*}
   \til{w}_{2,1}\til{w}_1^*\til{w}_{1,2}&=w_{2,1}w_1^*w_{1,2}+x_2^*x_1p_1x_1^*x_2
   =w_2+x_2^*x_1x_1^*x_2
   =w_2+x_2^*x_2
   =\til{w}_2.
  \end{align*}
  Analogously one can check that if $w_3$ satisfies the assumptions of in the first case, then $\til{w}_3=\til{w}_{3,2}\til{w}_2^*\til{w}_{2,1}\til{w}_1^*\til{w}_{1,3}.$
\end{proof}

We have gathered all necessary results to prove Theorem \ref{broise_improved}. The main idea (due to Broise) of its proof is to construct families of elements in the $\mathrm{II_1}$-factor satisfying the conditions in Lemma \ref{lem_bro2} (in case $n=3$) and to control their size. The novelty is only the control on the size, however, we need to repeat Broise's argument in order to get there.

\begin{proof}[Proof of Theorem \ref{broise_improved}:]
 Assume that $u\in\cU(\cM)$. By a standard argument (see for example \cite[Lemma 4]{Bro-67}) there exists a projection $p_0$ in maximal commutative von Neumann subalgebra containing $u$ such that $p_0\sim 1-p_0$. Since $p_0$ commutes with $u$, we have 
 $$u=(up_0+1-p_0)(p_0+u(1-p_0)).$$ 
 Put $u_0:=up_0$. By symmetry it is enough to show that $u_0+1-p_0$ is a product of 4 unitaries as in the statement of the theorem.
 
Let  $\{p_{0}(n)\}_{n\in\bbN_0}$ be a sequence of pairwise orthogonal projections satisfying
 $$p_{0}(0)=p_0,\quad \tau(p_{0}(n))=2^{-(n+1)},\quad \sum_{n\in\bbN_0}p_{0}(n)=1.$$
 Let $\cN_1$ denote the von Neumann algebra generated by $u_0$. There exist two orthogonal projections $p_1(1),p_2(1)\in\cN_1'\cap\mathcal{M}$ such that
 $$p_1(1)+p_2(1)=p_0(0),\quad \tau(p_1(1))=\tau(p_2(1))=\tau(p_0(1))=2^{-2}.$$
 Since the projections $p_0(1),p_1(1)$ and $p_2(1)$ are equivalent and pairwise orthogonal, there exists a family $\set{v_{i,j}(1)}_{0\leq i,j\leq 2}$ of elements in $\cM$ such that 
 $$v_{i,i}(1)=p_i(1),\quad v_{i,l}(1)v_{l,j}(1)=v_{i,j}(1),\quad (v_{i,j}(1))^*=v_{j,i}\textnormal{ for all }0\leq i,j,l\leq 2.$$
 Putting $u_1\coloneqq v_{0,1}(1)u_0v_{1,2}(1)u_0v_{2,0}(1)$, we obtain
 \begin{align*}
  u_1u_1^*&=(v_{0,1}(1)u_0v_{1,2}(1)u_0v_{2,0}(1))(v_{0,2}(1)u_0^*v_{2,1}(1)u_0^*v_{1,0}(1))\\
  &=v_{0,1}(1)u_0v_{1,2}(1)p_{2}(1)u_0u_0^*v_{2,1}(1)u_0^*v_{1,0}(1)\\
  &=v_{0,1}(1)u_0v_{1,2}(1)v_{2,2}(1)v_{2,1}(1)u_0^*v_{1,0}(1)\\
  &=v_{0,1}(1)v_{1,1}(1)v_{1,0}(1)
  =p_0(1)
  =u_1^*u_1.
 \end{align*}
 Inductively on can construct $\cN_n, p_1(n), p_2(n)$ and $ \set{v_{i,j}}_{0\leq i,j\leq n}, u_n,$
where $\cN_n$ is the von Neumann algebra generated by $u_{n-1}$, $p_1(n)$ and $p_2(n)$ are orthogonal projections in $\cM$ satisfying $p_1(n)+p_2(n)=p_0(n-1)$, $p_1(n)\sim p_2(n)\sim p_0(n),$ and $\set{v_{i,j}}_{0\leq i,j\leq n}$ is a family of elements in $\cM$ satisfying
 $$v_{i,i}(n)=p_i(n),\ v_{i,l}(n)v_{l,j}(n)=v_{i,j}(n),\ (v_{i,j}(n))^*=v_{j,i}(n)\tn{ for all }0\leq i,j,l\leq 2,$$
and 
$u_n\coloneqq v_{0,1}(n)u_{n-1}v_{1,2}(n)u_{n-1}v_{2,0}(n).$
For all $n\in\bbN$, we can assume that $\cN_n$ is commutative,
$p_1(n)$ and $p_2(n)$ belong to $\cN_n'$ (and hence commute with $u_{n-1}$) and $u_nu_n^*=u_n^*u_n=p_0(n)$.
Indeed, these properties have been verified for $n=1$ and can be verified inductively for higher $n$.

 Put
 $$w_{i,j}\coloneqq\sum_{m\geq 0}v_{i,j}(2m+1),\quad w_{i,j}'\coloneqq\sum_{m\geq 1}v_{i,j}(2m).$$
 Then $w_{0,0},\; w_{1,1},\; w_{2,2}$, respectively $p_0,\; w_{0,0}',\; w_{1,1}',\; w_{2,2}'$ are mutually orthogonal projections and
$w_{i,l}w_{l,j}=w_{i,j},\ w_{i,j}^*=w_{j,i},\ w_{i,l}'w_{l,j}'=w_{i,j}'$ and $(w_{i,j}')^*=w_{j,i}'.$
We define the following elements:
 \begin{alignat*}{3}
  w_0&\coloneqq\sum_{0\leq m}u_{2m+1}^*,\quad &w_1&\coloneqq\sum_{0\leq m}u_{2m}p_1(2m+1),\quad &w_2&\coloneqq\sum_{0\leq m}u_{2m}p_2(2m+1),\\
 w_0'&\coloneqq\sum_{1\leq m}u_{2m}^*,\quad &w_1'&\coloneqq\sum_{0\leq m}u_{2m+1}p_1(2m+2),\quad &w_2'&\coloneqq\sum_{0\leq m}u_{2m+1}p_2(2m+2).
 \end{alignat*}
 The equation $p_1(n+1)+p_2(n+1)=p_0(n)$ implies that
 $$w_1+w_2=\sum_{m\geq 0}u_{2m}(p_1(2m+1)+p_2(2m+1))=u_0+\sum_{m\geq 1}u_{2m}p_0(2m)=u_0+w_0^{'*},$$
 and
 $w_1'+w_2'=\sum_{m\geq0}u_{2m+1}p_0(2m+1)=w_0^*.$
 Using these two formulas as well as $u_nu_m=0$ for all $n\neq m$ (since $u_n=p_0(n)u_np_0(n)$), we obtain
 \begin{align*}
  (w_1+w_2+w_0)(p_0+w_1'+w_2'+w_0')&=(u_0+w_0+w_0'^*)(p_0+w_0^*+w_0')\\
  &=u_0+w_0w_0^*+w_0'^*w_0'\\
  &=u_0+\sum_{m\geq 1}p_0(m)=u_0+1-p_0.
 \end{align*}
We conclude that
 \begin{align*}
   (w_1+w_2+w_0)^*(w_1+w_2+w_0)
  =&\ w_1^*w_1+w_2^*w_2+w_0^*w_0\\
  =&\ \sum_{m\geq 0}p_0(2m)(p_1(2m+1)+p_2(2m+1))+\sum_{m\geq 0}p_0(2m+1)\\
  =&\ \sum_{m\geq 0}p_0(m)  =\ 1.
 \end{align*}
 Analogously one has $(w_1+w_2+w_0)(w_1+w_2+w_0)^*=1$. That is, $w_1+w_2+w_0$ is unitary. Similarly, $p_0+w_1'+w_2'+w_0'$ is unitary. Observe that
 \begin{align*}
  w_{0,2}w_2^*w_{2,1}w_1^*w_{1,0}
  &=\sum_{m\geq 0}v_{0,2}(2m+1)u_{2m}^*v_{2,1}(2m+1)u_{2m}^*v_{1,0}(2m+1)
  =w_0,
 \end{align*}
and similarly $w_0'=w_{0,2}'w_{2}'^*w_{2,1}'w_1'^*w_{1,0}'$. Hence we can apply Lemma \ref{lem_bro2}(ii) to obtain that $u_0+1-p_0$ is a product of four elements as described in the statement of the theorem. Thus, $u$ is a product of eight such elements.

We now turn to the bounds of the 2-norms -- which is the only novel aspect of this proof --  that were claimed in the statement of the theorem. Assume that $\norm{1-u}<\delta$. It is clear that $\norm{1-u_0-p_0^{\perp}}<\delta$. For $u_1=v_{0,1}(1)u_0v_{1,2}(1)u_0v_{2,0}$ we then get 
 \begin{align*}
  &\norm{p_0(1)-u_1}\\
  =&\norm{v_{0,1}(1)(p_0(0)-u_0)v_{1,2}(1)v_{2,0}(1) + v_{0,1}(1)u_0v_{1,2}(1)(p_0(0)-u_0)v_{2,0}(1)}\\
  \leq &\norm{v_{1,2}(1)v_{2,0}(1)}\cdot \norm{v_{0,1}(1)(p_0-u_0)} + \norm{v_{0,1}(1)u_0v_{1,2}(1)}\cdot\norm{(p_0-u_0)v_{2,0}(1)}\\
  \leq &\norm{v_{0,1}(1)(p_0-u_0)} + \norm{(p_0-u_0)v_{2,0}(1)}
  \leq \norm{p_0-u_0}\cdot\norm{p_0} + \norm{p_0-u_0}\cdot\norm{p_0}\\
  < & 2\delta.
 \end{align*}
 It follows by induction that for $u_n=v_{0,1}(n)u_{n-1}v_{1,2}(n)u_{n-1}v_{2,0}(n)$ we have
 $$\norm{p_0(n)-u_n}_2<2^{n}\delta.$$
 Now consider $w_1=\sum_{n\geq 0}u_{2n}p_1(2n+1)$, the other $w_i$'s can be treated similarly. 
 From the above estimate we conclude
 \begin{align*}
  \norm{\sum_{n\geq 0}p_1(2n+1) - \sum_{n\geq 0}u_{2n}p_1(2n+1)}_2 &\leq \sum_{n\geq 0}\norm{(1-u_{2n})p_1(2n+1)}_2\\
  &= \sum_{n\geq 0}\norm{(p_0(2n)-u_{2n})p_1(2n+1)}_2\\
  &\leq \sum_{n\geq 0}\norm{p_0(2n)-u_{2n}}\cdot\norm{p_1(2n+1)}_2\\
  &\leq \sum_{n\geq 0}\min\set{2,2^{2n}\delta}\cdot 2^{-(2n+3)}.
 \end{align*}
 That is, we have
 $$\norm{1-u_i}_2=\norm{\left( \begin{smallmatrix} 1 &0 &0 \\ 0 &1 &0\\ 0 &0 &1  \end{smallmatrix} \right) - \left( \begin{smallmatrix} 1 &0 &0 \\ 0 &w_1 &0\\ 0 &0 &w_1^{*}  \end{smallmatrix} \right)}_2 \leq \sum_{n\geq 0}\min\set{2,2^{2n}\delta}\cdot 2^{-(2n+2)}.$$

It remains to show that for arbitrarily small $\eps>0$ there exists $\delta>0$ such that $\norm{1-u}<\delta$ implies
$\norm{1-u_i}_2<\eps.$ Indeed, this follows from standard estimates, in fact 
$$\sum_{n\geq 0}\min\set{2,2^{2n}\delta}\cdot 2^{-(2n+2)} \leq c \cdot \delta |\log(\delta)|$$ for some universal constant $c$. 
\end{proof}

\section{Generalized projective \texorpdfstring{$s$}{s}-numbers}\label{sec_proj_s_numbers}

%

In this section we summarize some facts on generalized $s$-numbers for II$_1$-factor von Neumann algebras, collected mainly from \cite{FK-86}, and introduce so-called generalized projective $s$-numbers which will play a major role in our analysis. 
 
Throughout this section, let $\cM$ denote a II$_1$-factor von Neumann algebra acting on a Hilbert space $\cH$ with faithful, positive, normal, and unital trace $\tau$. 
Fack and Kosaki provide a more general framework in \cite{FK-86}, using $\tau$-measurable operators (which are special possibly unbounded operators affiliated with $\cM$), but for our purposes, it suffices to consider operators in $\cM$ itself.
All results and definitions in this section hold true for semi-finite factor von Neumann algebras with faithful normal semi-finite trace. 

The classical $s$-numbers of compact operators can be generalized in the following way. 
\begin{defn}
 Let $T\in\cM$ and $t \geq 0$. We define the \textbf{$t$-th generalized $s$-number}\index{generalized $s$-number} $\mu_t(T)$ of $T$ as
$\mu_t(T)\coloneqq\inf\set{\norm{Tp}\mid p\in\Proj(\cM)\tn{ such that }\tau(1-p)\leq t}.$
\end{defn}

The natural domain for $\mu_t$ is $[0,1]$, but we will frequently regard it as a function on $[0,\infty)$, extended by zero to the right. We list some  important properties of generalized $s$-numbers.
\begin{lem}[Fack-Kosaki, see \cite{FK-86}]\label{lem_prop_s_numbers} Let $\cM$ be a finite factor von Neumann algebra and let $x,y\in \cM$.
 \begin{enumerate}
  \item[(i)] The map $[0,1]\ni t\mapsto\mu_t(x)$ is non-increasing and right continuous. Moreover, $$\lim_{t\searrow 0}\mu_t(x)=\norm{x}\in[0,\infty].$$
  \item[(ii)] $\mu_t(x)=\mu_t(\abs{x})=\mu_t(x^*)$ and $\mu_t(\alpha x)=\abs{\alpha}\mu_t(x)$ for $t>0$ and $\alpha\in\bbC$.
  \item[(iii)] $\mu_t(x)\leq\mu_t(y)$ for $t>0$, if $0\leq x\leq y$.
  \item[(iv)] $\mu_t\left(f(\abs{x})\right)=f\left(\mu_t(\abs{x})\right),\ t>0,$ for any continuous increasing function $f$ on $[0,\infty)$ with $f(0)\geq 0$.
  \item[(v)] $\mu_{t+s}(x+y)\leq\mu_t(x)+\mu_s(y)$ for $s,t>0$.
  \item[(vi)] $\mu_{t+s}(xy)\leq\mu_t(x)\mu_s(y),\ s,t>0$.
 \end{enumerate}
\end{lem}

Clearly, if $x\in \cM$ and $p\in\Proj(\cM)$, then we have $\mu_t(xp)=0\tn{ for } t\geq\tau(p).$ The $p$-norms on $\cM$ have the expression 
$$\|x\|_p = \left(\int_{[0,1]} \mu_t(x)^p dt \right)^{1/p}.$$
The following standard Markov-type inequality will turn out to be useful.

\begin{lem}\label{lem_ineq}
Let $x \in \cM$. We have $\mu_t(x)\leq \|x\|_1/t$ for all $t>0$.
\end{lem}
\begin{proof}
Note that $\|x\|_1 =\tau(\abs{x})=\int_{[0,1]}\mu_t(x)dt$. Assume to the contrary that $\mu_{t_0}(x)>\|x\|_1/t_0$ for some $t_0>0$. Since $\mu_t$ is non-increasing in $t$, this implies $\mu_t(x)>\|x\|_1/t_0$ for all $t\in[0,t_0]$. Hence, by positivity of $\mu_t$,
 \begin{align*}
  \int_{[0,1]}\mu_t(x)dt\geq \int_{[0,t_0]}\mu_t(x)dt>\int_{[0,t_0]}\frac{\|x\|_1}{t_0}dt=\|x\|_1,
 \end{align*}
 a contradiction.
\end{proof}

%
%
%
Let $\cP\cU(\cM)$ denote the projective unitary group of $\cM$, i.e., $\cP\cU(\cM)=\cU(\cM)/{S^1}$, where ${S^1}$ denotes the center of $\cU(\cM)$.
We are going to develop the notion of generalized projective $s$-numbers and prove some useful properties of these. Some of these properties will be freely used in the following sections.

\begin{lem}\label{lem_length_cts}
 Let $\cM$ denote a II$_1$-factor. Let $x\in \cU(\cM)$. The function $\lambda \mapsto \mu_t(1-\lambda x)$ is 1-Lipschitz in $\lambda\in {S^1}$ for all $t\geq 0$.
\end{lem}
\begin{proof}
 Let $\eps >0$ be arbitrary. We claim that there exists $\delta >0$ such that $\abs{\lambda_1-\lambda_2}<\delta$ for $\lambda_i\in S^1$, implies $\abs{\mu_t(1-\lambda_1 x)-\mu_t(1-\lambda_2 x)}<\eps$. We may assume without loss of generality that $\inf_{\tau(1-p)\leq t}\norm{(1-\lambda_1 x)p}\geq \inf_{\tau(1-q)\leq t}\norm{(1-\lambda_2 x)q}$.
 \begin{align*}
  \abs{\mu_t(1-\lambda_1 x)-\mu_t(1-\lambda_2 x)}&= \abs{\inf_{\tau(1-p)\leq t}\norm{(1-\lambda_1 x)p}-\inf_{\tau(1-q)\leq t}\norm{(1-\lambda_2 x)q}}\\
  &\leq \norm{(1-\lambda_1 x)q_0}-\norm{(1-\lambda_2 x)q_0}\\
  &\leq \norm{(1-\lambda_1 x)q_0-(1-\lambda_2 x)q_0}\\
  &=\abs{\lambda_1-\lambda_2}
 \end{align*}
where $q_0$ is chosen such that it realizes $\inf_{\tau(1-q)\leq t}\norm{(1-\lambda_2 x)q}$. \end{proof}

\begin{defn}\label{def_proj_s_number}
 Let $\cM$ be a finite factor with faithful normal trace $\tau$. We define 
$$\ell_t(x)\coloneqq\inf_{\lambda\in {S^1}}\mu_t(1-\lambda x)\mbox{ for } t\geq 0,\ x\in\cM,$$
and call $\ell_t$ the \textbf{$t$-th generalized projective $s$-number of $x\in \cM$}\index{generalized projective $s$-number}.

For a projection $p\in\Proj(\cM)$ we denote the restriction of $\ell_t$ to $p\cM p$ by $\ell_t^{(p)}$, that is 
$$\ell_t^{(p)}(x)=\inf_{\lambda\in {S^1}}\mu_t(p-\lambda pxp)\mbox{ for } t\geq 0.$$
We call the smallest number $s=s(x)\in[0,\infty]$ such that $\ell_t(x)\neq 0$ if and only if $t\in[0,s)$ the \textbf{projective rank of $x$}\index{projective rank}.
\end{defn}
We choose the notation $\ell_t$ because it serves as some generalization of a length function in our context. 
One can think of $\ell_t(u)$, $u\in \cU(\cM)$, as a measure of the size of the spectrum of $x$ after cutting out a piece of size $t\geq 0$, which reduces the size of the spectrum of $x$ as much as possible.

It follows immediately from the definition that $\ell_t(x)=\ell_t(\xi x)$ for all $\xi\in{S^1}$ and $t\geq 0$. 
Observe that we have $\ell_t=0$ for $t\geq 1$. By Lemma \ref{lem_prop_s_numbers}(ii) we have $\ell_t(x)=\ell_t(x^*)$ for every $t\geq 0$ and $x\in\cM$.
Using Lemma \ref{lem_prop_s_numbers}(vi), we conclude 
$$\ell_t(gxg^*)=\inf_{\lambda\in {S^1}}\mu_t(g(1-\lambda x)g^*) \leq\inf_{\lambda\in {S^1}}\norm{g}\norm{g^*}\mu_t(1-\lambda x)=\ell_t(x)$$
for all $g\in \cP\cU(\cM)$, $x\in\cM$ and $t\geq 0$. Replacing  $x$ by $g^*xg$, we obtain that $\ell_t$ is invariant under conjugation, i.e.,
$$\ell_t(gxg^*)=\ell_t(x) \tn{ for all }t\geq 0.$$

Now let $p\in\Proj(\cM)\setminus\set{0}$ and assume that $x\in\cM$ commutes with $p$. Then we have 
$\ell_t^{(p)}(x)\leq\ell_t(x)\mbox{ for all }t\geq 0.$
Indeed, we have
\begin{align*}
 \ell_t^{(p)}(x) &=\inf_{\lambda\in {S^1}}\mu_t(p-\lambda pxp)
 =\inf_{\lambda\in{S^1}}\inf_{q\in\Proj(\cM),\tau(1-q)\leq t}\norm{p(1-\lambda x)pq}\\
 &\leq \inf_{\lambda\in{S^1}}\inf_{q\in\Proj(\cM),\tau(1-q)\leq t}\norm{p}\norm{(1-\lambda x)q}\\
 &= \inf_{\lambda\in{S^1}}\inf_{q\in\Proj(\cM),\tau(1-q)\leq t}\norm{(1-\lambda x)q} =\ell_t(x).
\end{align*}

\begin{lem}\label{lem_decrease}
 $\ell_{s+t}(xy)\leq\ell_s(x)+\ell_t(y)$ for all $x,y\in\cM$ and $s,t\geq 0$. In particular, $\ell_t$ is non-increasing in $t\geq 0$.
\end{lem}
\begin{proof}
 Since ${S^1}$ compact and since $\mu_t(1-\lambda x)$ is continuous in $\lambda\in {S^1}$, we can choose $\lambda' \in {S^1}$ such that $\ell_t(y)=\mu_t(1-\lambda' y)$. Using Lemma \ref{lem_prop_s_numbers}(i),(v), we obtain
 \begin{align*}
  \ell_{s+t}(xy)&=\ell_{s+t}(x \lambda' y)=\inf_{\lambda\in {S^1}}\mu_{s+t}((1-\lambda x)\lambda' y+(1-\lambda' y))\\
  &\leq \inf_{\lambda\in {S^1}}\mu_{s}((1-\lambda x)\lambda' y)+\mu_t(1-\lambda' y)
  =\inf_{\lambda\in {S^1}}\mu_{s}(1-\lambda x)+\ell_t(y) =\ell_s(x)+\ell_t(y).
 \end{align*}
 To see that $\ell_t$ is non-increasing in $t$, let $y=1$ and use that obviously $\ell_t(1)=0$ for all $t\geq 0$ to obtain $\ell_{s+t}(x)\leq\ell_t(x)$ for all $t\geq 0$.
\end{proof}

\begin{lem}\label{lem_right_cts}
 $\ell_t(x)$ is right continuous in $t\in[0,\infty]$, where $x\in\cM$.
\end{lem}
\begin{proof}
Fix arbitrary $t\geq 0$ and $\eps>0$. By Lemma \ref{lem_length_cts}, for every $t\geq 0$ we can choose $\lambda_t\in{S^1}$ which realizes $\inf_{\lambda\in{S^1}}\mu_t(1-\lambda x)$. Moreover, for every $\lambda$ we can choose a maximal $\delta_{\lambda}> 0$ such that 
 $$\mu_{t}(1-\lambda x)-\mu_{t+\delta}(1-\lambda x) < \eps \quad \mbox{for all } \delta<\delta_{\lambda},$$
since $\mu_t$ is right continuous in $t\in[0,\infty]$.

We claim that $\delta\coloneqq \inf_{\lambda\in{S^1}} \delta_{\lambda} > 0$. Assume to the contrary that $\delta=0$, i.e., there exists no $\delta'>0$ such that $\mu_t(1-\lambda x) - \mu_{t+\delta'}(1-\lambda x) <\eps$ for all $\lambda\in{S^1}$. Then there exist $\lambda\in S^1$ and sequences $(\lambda_n)_n$ and $(\delta_n)_n$, $\delta_n\coloneqq \delta_{\lambda_n}$, such that
 $\lambda_n\rightarrow \lambda$ and $\delta_n\rightarrow 0$ as $n\rightarrow \infty$, and
$\mu_t(1-\lambda_n x) - \mu_{t+\delta_n}(1-\lambda_n x) \geq \eps$ for all $n \in \bbN$. Hence, by Lemma \ref{lem_length_cts}, we obtain
$$\mu_t(1-\lambda x) - \mu_{t+\delta_n}(1-\lambda x) \geq \eps - 2 |\lambda - \lambda_n|.$$
On the other hand we have $\mu_t(1-\lambda x)-\mu_{t+\delta'}(1-\lambda x) < \eps/2$
for some $\delta'>0$. There exists $n_1\in\bbN$ such that $\delta_n \leq \delta'$ for all $n\geq n_1$. But this implies
 $$\eps/2> \mu_t(1-\lambda x)-\mu_{t+\delta'}(1-\lambda x) \geq \mu_t(1-\lambda x) - \mu_{t+\delta_n}(1-\lambda x) \geq \eps - 2 |\lambda - \lambda_n|.$$
Thus, we obtain a contradiction, when $n$ is large enough.
 Hence $\delta>0$ and
 $$\eps > \mu_t(1-\lambda_{t+\delta}x)-\mu_{t+\delta}(1-\lambda_{t+\delta}x) = \mu_t(1-\lambda_{t+\delta}x)-\ell_{t+\delta}(x)\geq \ell_t(x)-\ell_{t+\delta}(x).$$
 Since $t\geq 0$ and $\eps>0$ were arbitrary, we are done.
\end{proof}

The following proposition summarizes the above proven properties of generalized projective $s$-numbers.

\begin{prop}\label{prop_proj_s_numbers}
 Let $\cM$ be a finite factor with faithful normal, positive, and unital trace $\tau$. Let $x,y\in\cM$ and $u\in\cU(\cM)$. Let $p$ be a projection that commutes with $x$.
 \begin{enumerate}
  \item[(i)] $\ell_t(x)=\ell_t(x^*)$ for all $t\geq 0$.
  \item[(ii)] $\ell_t(x)=0$ for all $t\geq\tau(1)$.
  \item[(iii)] $\ell_t(uxu^*)=\ell_t(x)$ for all $t\geq 0$.
  \item[(iv)] $\ell_t^{(p)}(x)\leq\ell_t(x)$ for all $t\geq 0$.
  \item[(v)] $\ell_{s+t}(xy)\leq \ell_s(x)+\ell_t(y)$ for all $s,t\geq 0$.
  \item[(vi)] $\ell_t(x)$ is non-increasing in $t\geq 0$.
  \item[(vii)] $\ell_t(x)$ is right continuous in $t\geq 0$.
 \end{enumerate}
\end{prop}

The following proposition is the before mentioned easy observation on the length of products of conjugates. The proof is a straight forward application of the above properties of generalized projective $s$-numbers.

\begin{prop}\label{prop_easy_dir}
 If $u\in G\coloneqq\cP\cU(\cM)$ is a product of $k$ conjugates of $v\in G$ and $v^{-1}$, then $\ell_{k\cdot t}(u)\leq k\cdot\ell_t(v)$ for all $t\geq 0$.
\end{prop}
\begin{proof}
 Let $t\geq 0$. By assumption we can write $u=g_1^{}v^{\eps_1}g_1^*g_2^{}v^{\eps_2}g_2^*\cdots g_k^{}v^{\eps_k}g_k^*$ for some $g_i\in G$ and $\eps_i\in\set{1,-1}$, where $i=1,\ldots,k$ . Using $\ell_t(gwg)=\ell_t(w)$ for all $g,w\in G$ and that $\ell_t(w)=\ell_t(w^*)$ for all $t\geq 0$, we deduce
 \begin{align*}
  \ell_{kt}(u)&\leq\ell_t(g_1^{}v^{\eps_1}g_1^*)+\ell_{(k-1)t}(g_2^{}v^{\eps_2}g_2^*\cdots g_k^{}v^{\eps_k}g_k^*)\\
  &\leq\ell_t(v^{\eps_1})+\ell_t(g_2^{}v^{\eps_2}g_2^*)+\ell_{(k-2)t}(g_3^{}v^{\eps_3}g_3^*\cdots g_k^{}v^{\eps_k}g_k^*)\\
  &=\ell_t(v)+\ell_t(g_2^{}v^{\eps_2}g_2^*)+\ell_{(k-2)t}(g_3^{}v^{\eps_3}g_3^*\cdots g_k^{}v^{\eps_k}g_k^*)\\
&\ \vdots\\
  &\leq k\cdot\ell_t(v),
 \end{align*}
 which proves our claim.
\end{proof}

We will clarify in the next section that a converse to the preceding proposition (even up to a multiplicative constant) is too much to hope for.

\section{Bounded normal generation for factors of type I${}_n$}\label{sec_I_n}
Property {\rm (BNG)} for compact connected simple Lie groups (e.g. the projective unitary group $\cP\cU(n)$) has been proved in \cite{NS-12}, where a result much in the spirit of Theorem \ref{thm_sun_bdd} was proved. The main goal is to provide a rank-independent result of the same type, this is Theorem \ref{thm_sun} -- which is more subtle, since, instead of generating the unitary matrix eigenvalue by eigenvalue, we need to start and control a parallel generation process. We repair the rank-independent result \cite[Lemma 4.15]{ST-13} for $\cP\cU(n)$ and clarify in Remark \ref{rem_wrong} why this is necessary. Some arguments are borrowed from these articles but our path focusses on the $\cP\cU(n)$-case and our version of \cite[Lemma 4.15]{ST-13} as well as its proof differ considerably.

In this section we fix the following notation. Let $T$ denote the maximal torus of diagonal matrices in $\cU(n),\ 2\leq n\in\bbN$, i.e., 
$$T=\set{\diag(e^{\complex\theta_0},\ldots,e^{\complex\theta_{n-1}})\mid \theta_i\in[0,2\pi),\ i=0,1,\ldots,n-1}.$$ 
We decompose $T$ into $n$ subgroups $T_j,\ j=0,\ldots, n-1,$ which are defined as  
$T_0 \coloneqq \cZ(\cU(n)),$ and $T_j \coloneqq \set{\diag(1,\ldots,1,\lambda,\ldots,\lambda) \mid \lambda\in{S^1}}$, where $\lambda$ is on the positions $j+1,\ldots,n$.
Observe that every element $u=\diag(\lambda_0,\ldots,\lambda_{n-1})\in T$ can be decomposed uniquely into the product of commuting factors $u=u_0\cdot\ldots\cdot u_{n-1}$, where $u_i\in T_i$ with
\begin{align*}
 u_0&= \diag(\lambda_0,\lambda_0,\ldots,\lambda_0),\\
 u_i&= \diag(1,\ldots,1,\lambda_{i}\overline{\lambda}_{i-1},\lambda_{i}\overline{\lambda}_{i-1},\ldots,\lambda_{i}\overline{\lambda}_{i-1})\quad\mbox{ for }i\geq 1.
\end{align*}
We call this decomposition the \textbf{torus decomposition of $u$}\index{torus decomposition}.
Let us point out that when working in $\cP\cU(n)$, the factor $u_0$ in the decomposition $u=u_0\cdot\ldots\cdot u_{n-1}$ is left out since $u_0$ is central.

We will use another decomposition for $u\in\cSU(n)$ (respectively $u\in\cP\cU(n)$). For $j=0,\ldots, n-2$ let $S_j,$ denote the subgroup of $\cSU(n)$ of matrices of the form 
$$\begin{pmatrix} 1 & &0 \\
                  &\cSU(2) &\\
                 0 & &1\\
                 \end{pmatrix},$$
where the copy of $\cSU(2)$ sits at the entries $j+1$ and $j+2$. Then $u$ can be decomposed into factors $u_i\in S_i$, $i=0,\ldots,n-2$, where
\begin{align*}
 u_0&=\diag(\lambda_0,\overline{\lambda}_0,1,\ldots,1),\\
 u_i&=\diag(1,\ldots,1,\lambda_0\cdot\ldots\cdot\lambda_{i},\overline{\lambda}_0\cdot\ldots\cdot\overline{\lambda}_{i},1,\ldots,1).
\end{align*}
This decomposition is called the \textbf{product decomposition}\index{product decomposition}. Note that the factors in the torus decomposition as well as the product decomposition mutually commute.

We will need both of the above  decompositions in order to get the desired rank-independent result. The error that is hidden in \cite[Lemma 4.15]{ST-13} stems from an incorrect use of these decompositions.  To see that \cite[Lemma 4.15]{ST-13} is wrong see Remark \ref{rem_wrong}.
Let us now come to the first step in the proof of our rank-independent result. We will need \cite[Lemma 5.20]{NS-12} from work of Nikolov and Segal, which provides the basic building block for all finite-dimensional generation processes. For angles $\varphi \in \mathbb R/2\pi \mathbb Z$, we will frequently use the notation $|\varphi| := \min_n |\varphi + 2\pi n|$.

\begin{lem}[Nikolov-Segal]\label{lem_su2}
 Let $u=\left(\begin{smallmatrix}e^{\complex\varphi} &0\\ 0 &e^{-\complex\varphi}\end{smallmatrix}\right)$ and $v=\left(\begin{smallmatrix}e^{\complex\theta} &0\\ 0 &e^{-\complex\theta}\end{smallmatrix}\right)$ be non-central elements in $G\coloneqq \cSU(2)$ with $\theta \in [-\pi/2,\pi/2]$. If $\abs{\varphi}\leq m\abs{\theta}$ for some even $m\in\bbN$, then $u\in(v^G)^m$.
\end{lem}

Let us now analyze how to use the above lemma on a single factor $u_i\in S_i$ in $\cP\cU(n)$.

\begin{lem}\label{lem_single_gen}
 Let $G\coloneqq \cP\cU(n)$ with $n\geq 2,\ n\in\bbN$. Let $u=\diag(e^{\complex\varphi_0},\ldots,e^{\complex\varphi_{n-1}}),v=\diag(e^{\complex\theta_0},\ldots,e^{\complex\theta_{n-1}})\in G$ and assume that $u_0\cdot\ldots\cdot u_{n-1}$ with $u_i\in S_i$. If $\abs{\varphi_0+\ldots+\varphi_{i-1}}\leq m\abs{\theta_{j-1}-\theta_{j}}$ for some $i,j\in\set{1,\ldots,n-1}$ and even $m\in\bbN$ then 
 $$u_i\in (v^G\cup v^{-G})^{2m}.$$
\end{lem}
\begin{proof}
 Write $v_0\cdot\ldots\cdot v_{n-1},\ v_i\in T_i$ in its torus decomposition.
 Let $g\in S_j$ be the permutation swapping the diagonal entries at the positions $j,j+1$. Then $[v,g]=[v_j,g]\in S_j$. Let $h\in G$ such that $u_i^h\in S_j$. Using the given inequality $\abs{\varphi_0+\ldots+\varphi_{i-1}}\leq m\abs{\theta_{j-1}-\theta_{j}}$ (note that $\varphi_0+\ldots+\varphi_{i-1}$ is the angle of the first nontrivial eigenvalue of $u_i$) and Lemma \ref{lem_su2} we conclude 
 $$u_i\in h^{-1}([v_j,g]^{S_j}\cup [v_j,g]^{-S_j})^mh\subset (v^G\cup v^{-G})^{2m}.$$
 This concludes the proof.
\end{proof}

The following result is a crucial point in simultaneous generation with the help of $\cSU(2)$-copies. 
In the generation process we decompose the generating element $v$ into elements of the tori $T_i$, but the generated element $u$ needs to be decomposed into elements of $S_j$. 
\begin{lem}\label{lem_sun}
 Let $G\coloneqq \cP\cU(n)$, $n\geq 2,\ m\in\bbN$ even and $s\in\bbN_0$. Let 
$$u=\diag(e^{\complex\varphi_0},\ldots,e^{\complex\varphi_{n-1}})=u_0\cdot u_1\cdot\ldots\cdot u_{n-1}$$
 be the product decomposition of $u$ and let 
 $$v=\diag(e^{\complex\theta_0},\ldots,e^{\complex\theta_{n-1}})=v_0\cdot v_1\cdot\ldots\cdot v_{n-1}$$ with $v_i\in T_i$ be the torus decomposition of $v$. 
 For $0\leq k\leq s$ and $0\leq l\leq s$ let $i_k$ and $j_l$ be elements of $\set{0,\ldots,n-1}$.
 If $\abs{i_k-i_l},\abs{j_k-j_l}\geq 2$ for all $k\neq l$ and 
 $$\abs{\varphi_{0}+\varphi_1+\ldots+\varphi_{i_k}}\leq m\abs{\theta_{j_k}-\theta_{j_k+1}}\quad\mbox{ for }k,l=0,\ldots,s.$$
 Then, we obtain $u_{i_1}\cdot u_{i_2}\cdot\ldots\cdot u_{i_s}\in\left(v^G\cup v^{-G}\right)^{2m}.$ 
\end{lem}
\begin{proof}
 Write $v=v_{j_1}\cdot\ldots\cdot v_{j_s}\cdot {v''}$ where ${v''}$ commutes with $S_{i_1},\ldots,S_{i_s}$. Note that $S_{j_k}$ and $S_{j_l}$ commute elementwise for $k\neq l$. Moreover, ${v''}$ commutes with $S_{j_k}$ for all $k=1,\ldots,s$. Thus we get $$ \left(v_{}^{S_{j_1}\cdot\ldots\cdot S_{j_s}}\right)^m=\left( v_{j_1}^{S_{j_1}}\cdot\ldots\cdot v_{j_s}^{S_{j_s}}\cdot{v''}_{}^{S_{j_1}\cdot\ldots\cdot S_{j_s}}\right)^m=\left(v_{j_1}^{S_{j_1}}\right)^m\cdot\ldots\cdot\left( v_{j_s}^{S_{j_s}}\right)^m\cdot {v''}_{}^m.$$
 Let $g_{j_k}\in S_{j_k}$ be a permutation switching positions $j_k$ and $j_{k+1}$ for $k=0,\ldots,s$. Define
 $$g\coloneqq g_{j_1}\cdot\ldots\cdot g_{j_s}\in S_{j_1}\cdot\ldots\cdot S_{j_k}.$$ 
 Consider now the commutator $[v,g]=vgv^{-1}g^{-1}\in(v^G\cup v^{-G})^2$. Observe that $[v,g]\in S_{j_1}\cdot\ldots\cdot S_{j_s}$. Let $h\in G$ be a permutation such that $S_{i_k}^h=S_{j_k}$ for all $k=0,\ldots,s$. 
 Using Lemma \ref{lem_su2}, we obtain $u_{i_k}^h\in\left([v_{j_k},g_{j_k}]^{S_{j_k}}\cup [v_{j_k},g_{j_k}]^{-S_{j_k}} \right)^{m}$ for all $k=0,\ldots,s$, and hence
 $$u_{i_1}\cdot\ldots\cdot u_{i_s}\in h^{-1}\left(\left([v,g]^{S_{j_1}\cdot\ldots\cdot S_{j_s}}\right)^m\right)h\subset (v^G\cup v^{-G})^{2m}.$$
 This completes the proof.
\end{proof}

In order to have a relation between projective $s$-numbers we need that for all $\theta\in[-\pi,\pi]$, one has $\abs{\theta}/2\leq\sqrt{2(1-\cos\theta)}\leq\abs{\theta}$.
%
%
The following definition is crucial in order to obtain estimates between projective $s$-numbers and eigenvalue differences, which in turn will be compared to angles. 

\begin{defn} \label{def_optimal}
Assume that $u\in G\coloneqq \cU(n)$, $2\leq n\in\bbN$. Let us say that $\til{u}=\diag(\lambda_0,\ldots,\lambda_{n-1})\in u^G\cap T$ is \textbf{optimal}\index{optimal} if 
\begin{enumerate}
 \item $\abs{\lambda_0-\lambda_{1}}\geq\abs{x_0-x_1}$ for all $v=\diag(x_0,\ldots,x_{n-1})\in u^G\cap T$;
 \item $\abs{\lambda_i-\lambda_{i+1}}=\abs{x_i-x_{i+1}}$ for all $i=0,\ldots,k-1$ implies $\abs{\lambda_{k}-\lambda_{k+1}}\geq \abs{x_{k}-x_{k+1}}$.
\end{enumerate}
\end{defn}
This defines a lexicographic order on the eigenvalue differences, hence for every $u\in \cU(n)$ respectively $\cP\cU(n)$, there exists an optimal element $\til{u} \in u^G \cap T$. For two different optimal elements $\til{u},{v}$, we have $\abs{\lambda_i-\lambda_{i+1}}=\abs{x_i-x_{i+1}}$ for all $i\in\set{0,\ldots,n-2}$. Note that the eigenvalue differences need not be ordered, i.e., it may happen that for $|\lambda_i -\lambda_{i+1}| < |\lambda_{i+1} -\lambda_{i+2}|$ even if $u$ is optimal. However, for an optimal element $u$ there exists a permutation $\sigma\in S_X$, where $X=\set{0,\ldots,n-2}$ and $S_X$ denotes the group of permutations on $X$, such that  $$\abs{\lambda_{\sigma(0)}-\lambda_{\sigma(0)+1}}\geq\abs{\lambda_{\sigma(1)}-\lambda_{\sigma(1)+1}}\geq\ldots\geq\abs{\lambda_{\sigma(n-2)}-\lambda_{\sigma(n-2)+1}}.$$
We call such a permutation a \textbf{permutation associated to} the optimal element $u$. Note that our definition of optimality differs slightly from the one given in \cite{ST-13}.

\begin{lem}\label{lem_singularcage}
 Let $u=\diag(\lambda_0,\ldots,\lambda_{n-1})\in T\subset\cU(n)$ be optimal and $\sigma$ a permutation such that $\abs{\lambda_{\sigma(i)}-\lambda_{\sigma(i)+1}}$ is monotone decreasing in $i=0,\ldots,n-2$, where $n\geq 2$, $n\in\bbN$. Then there exists $\lambda\in S^1$ (in fact $\lambda=\lambda_{n-1}$) such that 
$$\frac{1}{2}\abs{\lambda_{\sigma(2i)}-\lambda_{\sigma(2i)+1}}\leq\ell_i(u) \leq \mu_i(1-\lambda u) \leq \abs{\lambda_{\sigma(i)}-\lambda_{\sigma(i)+1}}\quad\textnormal{ for all }i=0,\ldots,n-2.$$
\end{lem}
\begin{proof}
To prove the first inquality, let $z_0=\diag(z,\ldots,z)\in\cZ(\cU(n))$ be arbitrary and fix a permutation $\tau\in S_Y$, $Y\coloneqq \set{0,\ldots,n-1}$, such that $\abs{z-\lambda_{\tau(0)}}\geq\abs{z-\lambda_{\tau(1)}}\geq\ldots\geq\abs{z-\lambda_{\tau(n-1)}}.$ Assume to the contrary, that $\abs{\lambda_{\sigma(2i)}-\lambda_{\sigma(2i)+1}}> 2\abs{z-\lambda_{\tau(i)}}$. Hence 
  $$\abs{z-\lambda_{\sigma(k)}}+\abs{z-\lambda_{\sigma(k)+1}} \geq\abs{\lambda_{\sigma(k)}-\lambda_{\sigma(k)+1}}>  2\abs{z-\lambda_{\tau(i)}}$$
  for all $k=0,\ldots,2i$ by the choice of $\sigma$. This implies $\sigma(k)\in\set{\tau(0),\ldots,\tau(i-1)}$ or $\sigma(k)+1\in\set{\tau(0),\ldots,\tau(i-1)}$.
Since $\set{\tau(0),\ldots,\tau(i-1)}$ contains $i$ elements but one case appears at least for $i+1$ times, we arrive at a contradiction. Since $z_0$ was chosen arbitrarily, the first inequality follows.
 
The inequality in the middle is obvious. To see the third inequality, let $\tau$ be a permutation such that $$\abs{\lambda_{n-1}-\lambda_{\tau(0)}}\geq\abs{\lambda_{n-1}-\lambda_{\tau(1)}}\geq\ldots\geq\abs{\lambda_{n-1}-\lambda_{\tau(n-2)}}.$$ By optimality of $u$, we have $\abs{\lambda_i-\lambda_{i+1}}\geq\abs{\lambda_i-\lambda_j}$ for all $j\geq i+1$. Hence, for all $\tau(i)=0,\ldots,n-2$,
 $$\ell_i(u)\leq\mu_i(\lambda_{n-1}-u)=\abs{\lambda_{n-1}-\lambda_{\tau(i)}}\leq\abs{\lambda_{\tau(i)}-\lambda_{\tau(i)+1}},$$
 while for $\tau(i)=n-1$, we get $\ell_i(u)= 0$. Thus for each $i$, $\ell_i(u)$ can be estimated from above by $\abs{\lambda_{\tau(i)}-\lambda_{\tau(i)+1}}$. Since both $\abs{\lambda_{\sigma(i)}-\lambda_{\sigma(i)+1}}$ and $\mu_i(\lambda_{n-1}-u)$ are decreasing in $i$, we obtain 
$\mu_i(\lambda_{n-1}-u)\leq\abs{\lambda_{\sigma(i)}-\lambda_{\sigma(i)+1}}.$ Setting $\lambda := \lambda_{n-1}$, this finishes the proof.
\end{proof}
\begin{rem}
 The first inequality in Lemma \ref{lem_singularcage} holds for any diagonal unitary.
\end{rem}

The above lemma implies the following important corollary which relates projective singular values and angles of elements in $\cU(n)$.

\begin{cor}\label{cor_length_angles}
  Let $u=\diag(e^{\complex\theta_0},\ldots,e^{\complex\theta_{n-1}}),\ v=\diag(e^{\complex\gamma_0},\ldots,e^{\complex\gamma_{n-1}})\in T\subset\cU(n)$ be optimal with associated permutation $\sigma,\ \tau$. Then $\ell_{ki}(u)\leq m\ell_i(v)$ for all $i=0,\ldots,n-1$ and some $k,m\in\bbN$ implies $\abs{\theta_{\sigma(2ki)}-\theta_{\sigma(2ki)+1}}\leq 4m\abs{\gamma_{\tau(i)}-\gamma_{\tau(i)+1}}$ for all $i=0,\ldots,n-1$, where we set $\theta_{i}=\gamma_i=0$ for all $i\geq n$.
\end{cor}
\begin{proof}
We conclude that 
$\abs{e^{\complex\theta_{\sigma(2kj)}}-e^{\complex\theta_{\sigma(2kj)+1}}}\leq 2m\abs{e^{\complex\gamma_{\tau(j)}}-e^{\complex\gamma_{\tau(j)+1}}}.$
 Using now the estimates 
$\abs{1-e^{\complex(\theta_{\sigma(2kj)}-\theta_{\sigma(2kj)+1})}} \leq \abs{\theta_{\sigma(2kj)}-\theta_{\sigma(2kj)+1})}$, and
$\abs{1-e^{\complex(\gamma_{\tau(2kj)}-\gamma_{\tau(2kj)+1})}} \geq \abs{\gamma_{\tau(2kj)}-\gamma_{\tau(2kj)+1})}/2$
 we obtain the claimed inequality.
\end{proof}

We need the following standard combinatorial lemma to control sums of angles (occuring in the $\cSU(2)$ product decomposition of an element in $\cSU(n)$, respectively $\cP\cU(n)$, rank-independently.

\begin{lem}\label{lem_sum_angles}
 Let $n\in\bbN$. Assume that $\alpha_1,\ldots,\alpha_n\in\bbR$ satisfy $\sum_{i=1}^n\alpha_i=0$. Then there exists a permutation $\sigma\in S_n$ such that for every $k\in\set{1,\ldots,n}$ one has
 $$\abs{\sum_{i=1}^k\alpha_{\sigma(i)}}\leq \max_{i=1,\ldots,n}\abs{\alpha_i}.$$
\end{lem}

\begin{proof}
 Without loss of generality we have $\alpha_1=\max_{i=1,\ldots,n}\abs{\alpha_i}>0$ and $\alpha_i\neq 0$ for all $i=1,\ldots,n$. Moreover, we may assume that $\alpha_1\geq\ldots\geq\alpha_l>0$ and $\alpha_{l+1}\leq \ldots\leq\alpha_n<0$ for some $l< n$. We now construct the permutation $\sigma\in S_n$. We let $\sigma(1)\coloneqq 1$ and $\sigma(2)=l+1$. Then $\alpha_{\sigma(1)}+\alpha_{\sigma(2)}\geq 0$.  
(i) Let $1\leq j_1\leq n$ be the unique smallest number such that 
 $$\alpha_1+\alpha_{l+1}+\ldots+\alpha_{l+j_1}< 0.$$
 Set $\sigma(1+i)\coloneqq l+i$, where $i=1,\ldots,j_1$. 
(ii) If there are no $\alpha_i$'s left, then we are done. Else, we let $1\leq j_2\leq l$ be the unique smallest number such that
 $$\alpha_1+\alpha_{l+1}+\ldots+\alpha_{l+j_1}+\alpha_{2}+\ldots+\alpha_{1+j_2} > 0.$$
Put $\sigma(1+j_1+i)\coloneqq 1+i$ for $i=1,\ldots,j_2$.
We obviously have for $k\leq 1+j_1+j_2$  and $$\abs{\sum_{i=1}^k \alpha_{\sigma(i)}}\leq \max_{i=1,\ldots,n}\abs{\alpha_i}= \alpha_{\sigma(1)}.$$
 Proceed inductively interchanging steps (i) and (ii) until $\sigma$ is defined on $\set{1,\ldots,n}$. This finishes the proof.
\end{proof}

\begin{defn}
 Let $u=\diag(e^{\complex\theta_0},\ldots,e^{\complex\theta_{n-1}})\in\cSU(n)$ such that $\sum_{i=0}^{n-1}\theta_i=0$. Let $\sigma\in S_n$ be as in Lemma \ref{lem_sum_angles}. Then we say that the element $\diag(e^{\theta_{\sigma(0)}},\ldots,e^{\complex\theta_{\sigma(n-1)}})$ \textbf{angle sum optimal}\index{angle sum optimal}. The permutation $\sigma$ is said to be \textbf{associated to} the angle sum optimal element $u$.
\end{defn}

The sole purpose of the following lemma is to control the angles of the first factor in the product decomposition of a unitary in ${\rm PU}(n)$.

\begin{lem}\label{lem_opt}
 Assume that $u=\diag(e^{\complex\theta_0},\ldots,e^{\complex\theta_{n-1}})\in {\rm PU}(n)$ is optimal. 
 Then we have 
 $$2\abs{\theta_0-\theta_1}\geq \max_{i=0,\ldots,n-1}\abs{\theta_i}.$$
\end{lem}
\begin{proof}
This is obvious.
\end{proof}

For Lie group $\cP\cU(n)$ we obtain the following rank-dependent result by successive application of Lemma \ref{lem_single_gen}.

\begin{thm}\label{thm_sun_bdd}
 Let $G\coloneqq \cP\cU(n),\ n\geq 2,$ and assume that $u,v\in G\setminus\set{1}$ satisfy $\ell_{0}(u)\leq m\ell_{0}(v)$ for some $m\in\bbN$. Then
 $$ u\in (v^G\cup v^{-G})^{8mn}.$$
\end{thm}
\begin{proof}
 Without loss of generality, 
$u=\diag(e^{\complex \theta_0},\ldots,e^{\complex\theta_{n-1}})=u_0\cdot\ldots\cdot u_{n-2}\in S_0\cdot\ldots\cdot S_{n-2}$
with $\sum_{i=0}^{n-1} \theta_i =0$ is angle sum optimal with associated permutation $\sigma$ and 
$v=\diag(e^{\complex\gamma_0},\ldots,e^{\complex\gamma_{n-1}})$ is optimal with associated permutation $\tau$. Since $\ell_0(u)\leq m\ell_0(v)$, we conclude from Corollary \ref{cor_length_angles} and the definition of optimality that for all $i=0,\ldots,n-1$ we have 
 $$4m\abs{\gamma_{\tau(0)}-\gamma_{\tau(0)+1}}\geq \max_{j \neq k}\abs{\theta_{j}-\theta_k}\geq \abs{\theta_{\sigma(0)}+\theta_{\sigma(1)}+\ldots+\theta_{\sigma(i)}}.$$
 Note that for $i=0$ we use Lemma \ref{lem_opt} and that $\ell_0(u)\leq m\ell_0(v)$ implies already $\abs{\theta_i-\theta_j}\leq m\abs{\gamma_0-\gamma_1}$ for all $i,j$ by monotonicity of the correspondence between angles and projective $s$-numbers.
 Now we can apply Lemma \ref{lem_single_gen} for each $u_i$ and hence obtain 
$u_i\in (v^G\cup v^{-G})^{8m}.$
 Proceeding the same way for all $u_i$'s we have
$u\in (v^G\cup v^{-G})^{8mn}.$ This finishes the proof.
\end{proof}

\begin{rem}
 Theorem \ref{thm_sun_bdd} can actually be sharpened in the sense that one does not need the conjugacy class of $v^{-1}$. To see this, observe that one may choose $n-1$ permutations $\pi_1,\ldots,\pi_{n-1}\in G$ (for example $\pi_i(k)=k+i$ modulo $n$) such that 
 $$v\pi_1 v\pi_1^{-1}\cdot\ldots\cdot\pi_{n-1} v\pi_{n-1}^{-1}=\diag(e^{\complex\gamma_0}\cdot\ldots\cdot e^{\complex\gamma_{n-1}},\ldots,e^{\complex\gamma_0}\cdot\ldots\cdot e^{\complex\gamma_{n-1}})=1.$$
 Thus $1\in (v^G)^{n}$, which implies $v^{-1}\in (v^G)^{n-1}$. 
\end{rem}

\begin{cor}\label{cor_bdd_rk}
 Assume that $v\in G\coloneqq \cP\cU(n)$ is non-trivial, where $n\geq 2$. Then for every $k\geq 16n/\ell_0(v)$ we have
$G=(v^G\cup v^{-G})^{k}.$
 In particular, $\cP\cU(n)$ has property {\rm (BNG)}.
\end{cor}
\begin{proof}
 Since $v$ is non-trivial we have $\ell_0(v)>0$. It is trivial that for any $u\in G$ one has $\ell_0(u)\leq \frac{2}{\ell_0(v)}\ell_0(v)=2$. Using Theorem \ref{thm_sun_bdd} we conclude
$u\in (v^G\cup v^{-G})^{8n\cdot \lceil 2/\ell_0(v)\rceil}.$ Since $u$ was arbitrary, $G$ has property {\rm (BNG)}.
\end{proof}

The preceding result is rank-dependent, in the sense that the assumptions are purely spectral, but the size of the exponents depends on the rank of the group. On the other side, ignoring the constant it is clearly optimal in the sense that any normal generation function is bounded from below by $C/\ell_0(v)$ for some constant.

Now we come to the main result of this section. The main ingredient is Lemma \ref{lem_sun}, which we apply simultaneously at various diagonal entries. The result is rank independent, since the exponents depends only on the fraction $s/n$ of eigenvalues which satisfy some spectral assumption. This will be crucial when we will generalize the theory to II${}_1$-factors.

\begin{thm}\label{thm_sun}
 Let $G\coloneqq \cP\cU(n)$, where $n \geq 2$.
 Assume that $u,v\in G$ satisfy $\ell_{0}(u)\leq m\ell_t(v)$ for some $m\in\bbN$ and $t=0,1,\ldots,s-1\leq (n-1)/2$. Then $$u\in (v^G\cup v^{-G})^{24m\lceil n/s\rceil}.$$
\end{thm}
\begin{proof}
 Since we are in $\cP\cU(n)$ we may assume, multiplying with a central element if necessary, that the angle sums of $u$ and $v$ add up to $0$. Without loss of generality $u$ is angle sum optimal and $v$ is  optimal with associated permutation $\sigma$ and $\tau$ respectively. 
 The first step is to generate most of $u=u_0\cdot\ldots \cdot u_{n-2}$ (in the product decomposition) simultaneously. 
 Assume that $n-1$ is divisible by two (if not, the following works equally well for $n-2$ instead since we are generous with the number of conjugates). We split the set $A\coloneqq \set{0,\ldots,n-2}$ of indices into two sets $A_i\subset A$ with cardinality $(n-1)/2$ and such that $\abs{a-b}\geq 2$ for any distinct $a,b\in A_i$, $i=1,2$. Let $N$ denote the unique largest integer divisible by $s$ such that $N\leq\frac{n-1}{2}$. Further decompose each $A_i$ into $2N/s$ sets $A_{i,j}$ of cardinality $s/2$. Then $A_i\setminus \bigcup_{l=1,\ldots,2N/s}A_{i,l}$ has at most $s-1$ elements. 
Let $B\coloneqq\bigcup_{i=1,2\ l=1,\ldots,2N/s}A_{i,l}$ and observe that the cardinality of $A\setminus B$ is at most $2(s-1)$. 
By Corollary \ref{cor_length_angles} (and Lemma \ref{lem_opt} for the case $j=0$), for all $j=0,\ldots,s-1,$ we have
 $$\abs{\sum_{i=0}^j\theta_{\sigma(i)}}\leq 4m\abs{\gamma_{\tau(j)}-\gamma_{\tau(j+1)}}.$$ 
Applying now Lemma \ref{lem_sun} to all $4N/s$ sets $A_{i,l}$ we have 
 $$\prod_{j\in B}u_j \in (v^G\cup v^{-G})^{8m\cdot 4N/s}.$$
 Using again Lemma \ref{lem_sun} for the remaining factors of $u$ we obtain 
 $$\prod_{j\in A\setminus B}u_j = \prod_{j\in A_1\setminus B}u_j \prod_{j\in A_2\setminus B}u_j \in (v^G\cup v^{-G})^{2\cdot 8m}.$$
 Thus from $2N/s + 1 \leq 3N/s$ we conclude
$u\in (v^G\cup v^{-G})^{48 mN/s} \subseteq (v^G\cup v^{-G})^{24 m\lceil n/s\rceil}.$
\end{proof}

Let us end this section with an example that shows that some claims in \cite{ST-13} were too optimistic.

\begin{rem}\label{rem_wrong}
 Let $u=\diag(\lambda^{-n-1},\lambda,\lambda,\ldots,\lambda)$ and $v=\diag(\mu^{-n-1},\mu,\mu,\ldots,\mu)$ be non-trivial elements in $\mathrm{SU}(n)$ and consider the natural images $\bar u,\bar v \in \mathrm{PU}(n)$. We set $G:=\mathrm{PU}(n)$. Assume that $\arg(\lambda)/\arg(\mu)$ is irrational.
We claim that if $\bar u\in (\bar v^G\cup \bar v^{-G})^{k}$, then $k\geq n-1$. 

 Assume that $\bar u\in (\bar v^G\cup \bar v^{-G})^{k}$, i.e., $\bar u=g_1\bar v^{\eps_1}g_1^{-1}\cdot\ldots\cdot g_k\bar v^{\eps_k}g_k^{-1}$ with $g_i\in G,\ \eps_i\in\set{1,-1}$. Then lifting everything to $\mathrm{SU}(n)$, we get that $zu$ is a product of $k$ conjugates of $v^{\pm 1}\in\mathrm{SU}(n)$ for some $z\in\mathcal{Z}(\mathrm{SU}(n))$.  
 Hence $u'\coloneqq \mu^{-l}zu$ is a product of $k$ conjugates of $v'\coloneqq \mu^{-1}v$ in $\cU(n)$ for some $-k \leq l \leq k$. Now $\mu^{-1}v$ is a rank one perturbation of the identity and thus $u'$ is at most a rank $k$ perturbation of the identity in $\cU(n)$. But $n-1$ diagonal entries of $u'$ are of the form $\lambda \mu^{-l}z$, which is different from $1$ by the irrationality assumption. Hence $1-u'$ has rank at least $n-1$ and this implies $k\geq n-1$.
 
The example shows that an assumption of the form $\ell_{mi}(u) \leq m \ell_{i}(v)$ will not be enough to conclude that conjugates of $v,v^{-1}$ will generate $u$ in a number of steps that only depends on $m$ and not on $n$.
\end{rem}

\section{Dense products of conjugacy classes}\label{sec_II_1}

This section deals with preliminary material that is needed to prove property (BNG) for projective unitary groups of a $\mathrm{II_1}$-factors, endowed with the strong operator topology.

The strategy to extend our results from matrices to II${}_1$-factors is to approximate both $u$ and $v$ arbitarily close with elements having finite spectrum and rational weights and conjugate them to the same subgroup $\cP\cU(n)$. This allows us to use Theorem \ref{thm_sun}. Letting the approximation be finer and finer and using Lemma \ref{lem_epsconj} we conclude as a first intermediate result that $u$ is in the strong closure of a product of conjugates of $v$ and $v^{-1}$.
We need the following elementary approximation result. 
\begin{prop}\label{prop_WvNBVD}
  Assume that $u\in\cU(\cM)$ and $\eps>0$. There exists an element $u'\in\cU(\cM)$ having finite spectrum and corresponding spectral projections of rational trace such that $\norm{u-u'}_2<\eps.$
\end{prop}

\begin{proof}
 Choose pairwise distinct elements $\lambda_1,\ldots,\lambda_n\in{S^1},\ n\geq 2,$ such that for every $\lambda\in\sigma(u)$ there exists $i\in\set{1,\ldots,n}$ such that $\abs{\lambda-\lambda_i}<\eps/6$, $\arg(\lambda_i)<\arg(\lambda_{i+1})\mod 2\pi$ and every $\lambda_i$ has distance less than $\eps/6$ from $\sigma(u)$. Denote by $p_u$ the spectral measure of $u$ and define 
 $$p_i\coloneqq p_u(\set{\lambda \mid \arg(\lambda)\in[\arg(\lambda_i),\arg(\lambda_{i+1}))})$$
 for $i=1,\ldots,n-1$ and $p_n\coloneqq p_u([\lambda_n,\lambda_{1}))$. Note that $\sum_{i=1}^n p_i=1$. For $i=1,\ldots,n$ let $q_i$ be a subprojection of $p_i$ with rational trace and such that $\norm{p_i-q_i}_2<\eps/6n$. Set $q_0\coloneqq 1-\sum_{i=1}^{n}q_i$ and note that $\tau(q_0) \leq \varepsilon/6$.
Now set $u'\coloneqq q_0 + \sum_{i=1}^n \lambda_iq_i$. 
Hence we obtain
 \begin{align*}
  \norm{u-u'}_1 &=\norm{- q_0 + \sum_{i=1}^n(u p_i -\lambda_i q_i)}_1
  \leq \norm{q_0}_1 + \sum_{i=1}^n\norm{uq_i -\lambda_iq_i}_1 + \norm{q_i - p_i}_1\\
  &\leq \varepsilon/3 + \sum_{i=1}^n\norm{uq_i -\lambda_iq_i}\cdot \norm{q_i}_1
  < \varepsilon/3 + \sum_{i=1}^n\eps/6\cdot \norm{q_i}_1 =\eps/2.
 \end{align*}
Thus we have $\norm{u-u'}_2\leq 2\cdot \norm{u-u'}_1< \eps,$ as desired.
\end{proof}

%
%
The proof of Theorem \ref{thm_II_1} uses the following technical lemma, which allows us estimate singular values for sufficiently close $2$-norm approximations of a given element in a $\mathrm{II_1}$-factor $\cM$.

\begin{lem}\label{lem_useful_II_1} Let $u\in \cU(\cM)$ be noncentral. There exists $\varepsilon>0$ such that if $u' \in \cU(\cM)$ with $\norm{u-u'}_2<\eps$, then
$\ell_{2t}(u)\leq 2\ell_t(u') \mbox{ for all } t\geq 0.$
\end{lem}
\begin{proof}
Let $s>0$ denote the projective rank of $x$. Put $\delta\coloneqq \ell_{3s/4}(u)/3>0$.  Assuming $\eps$ to be small enough we obtain $\ell_{s/2}(u')\geq 2\delta>0$.
Indeed, $\ell_{3s/4}(u) \leq \ell_{s/2}(u') + \mu_{s/4}(u-u')$ and if
$\|u-u'\|_2 < \varepsilon$, then $\|u-u'\|_1 < \varepsilon$ and $\mu_{s/4}(u-u') < 4\varepsilon/s$ by Lemma \ref{lem_ineq}. Thus, any $\varepsilon>0$ with $\varepsilon \leq \delta s/4$ works for this purpose.

Right continuity (see Lemma \ref{lem_right_cts}) implies that there exists $\delta_0\in(0,s/2]$ such that 
$\ell_0(u)-\ell_t(u)\leq\delta$ for all $t\in[0,\delta_0]$.
Thus for all $t\in[0,\delta_0]$ we conclude
 \begin{align*}
  \ell_0(u) &\leq\ell_{\delta_0}(u)+\delta 
  \leq \ell_{\delta_0/2}(u')+\mu_{\delta_0/2}(u-u')+\delta
  \leq \ell_{t/2}(u')+\frac{2\eps}{\delta_0}+\delta.
 \end{align*}
 Thus if $\eps$ is small enough, we have $2\eps/\delta_0< \delta$ and hence
 $\ell_{2t}(u) \leq \ell_0(u)\leq  \ell_t(u')+2\delta \leq \ell_t(u')+ \ell_{s/2}(u')\leq 2\ell_t(u')$ for all $t\in[0,\delta_0].$
 For $t\in [\delta_0,s/2]$ we obtain 
 $\ell_{2t}(u)\leq\ell_t(u')+\mu_t(u-u')\leq \ell_t(u')+\frac{\eps}{\delta_0}\leq 2\ell_t(u').$ Thus for all $t\geq 0$ we get
$\ell_{2t}(u)\leq 2\ell_t(u'),$ as claimed.
\end{proof}

Assume now that $u,v\in G\coloneqq\cP\cU(\cM)$ satisfy $\ell_{0}(u)\leq m\ell_t(v)$ for all $t\in [0,s]$ and some $m\in\bbN$. We want to show that under these circumstances we have
$$u\in \overline{(v^G\cup v^{-G})^{48m \lceil 1 / s \rceil}}^{\norm{\cdot}_2}.$$
Let $\eps>0$ be small enough such that the assertion in Lemma \ref{lem_useful_II_1} holds. By Proposition \ref{prop_WvNBVD} there exist elements $u',v'\in\cU(\cM)$ such that $\norm{u-u'}_2,\ \norm{v-v'}_2 <\eps$ and $u'=\sum_{i=1}^n\lambda_ip_i,\  v'=\sum_{j=1}^m\zeta_jq_j$, where $\lambda_i,\zeta_j\in {S^1}$ and $p_i,q_j\in\Proj(\cM)$ satisfy $\tau(p_i)=r_i / s_i, \ \tau(q_j)=r_{j+n} / s_{j+n}$ for some $r_k,s_k\in\bbN\setminus\{0\}$.\\
Let $s_0$ denote the least common multiple of $s_1,\ldots,s_{n+m}$. Take subprojections $p_i'$ of the $p_i$ and $q_j'$ of the $q_j$ such that $\tau(p_i')=\tau(q_j')=1/s_0$ and 
$$u'=\sum_{i=1}^{s_0}\lambda_i'p_i',\quad v'=\sum_{j=1}^{s_0}\zeta_j'q_j',$$
where multiplicities are taken into account. It is easy to see that there exists $w\in\cU(\cM)$ such that $wp_i'w^*=q_i'$ for all $i=1,\ldots,s_0$, 
and we obtain that 
$$wu'w^*=\sum_{i=1}^{s_0}\lambda_i'q_i',\quad v'=\sum_{i=1}^{s_0}\zeta_i'q_i'.$$ 
We assume that $u'$ is such that we have $\ell_0(u')\leq\ell_0(u)$. Note that this is always possible by choosing the eigenvalues of the approximating element $u'$ such that 
$$\sup_{\lambda,\zeta\in\sigma(u')}\abs{\lambda-\zeta}\leq \sup_{\lambda,\zeta\in\sigma(u)}\abs{\lambda-\zeta}.$$
Using Lemma \ref{lem_useful_II_1} for $v,v'$ and the assumption that $\eps$ is sufficiently small we obtain for all $t\in[0,s/2]$ that
\begin{align}\label{eq_II_1_top_ineq}
 \ell_{0}(wu'w^*)&\ = \ell_{0}(u') \leq \ell_{0}(u) \leq m\ell_{2t}(v) < 2m\ell_t(v').
 \end{align}
We may assume that $s$ is rational. Indeed, if $s$ is irrational, using right continuity and the fact that the inequality $\ell_{2t}(v) < 2\ell_t(v')$ for $t\in[0,s/2]$ is strict, we replace $s$ by some rational $\tilde{s}>s$ such that $\ell_{2t}(v)\leq 2\ell_t(v')$ for all $t\in[0,\tilde{s}/2]$. Replacing $s_0$ by a multiple of $s_0$ if necessary, we may also assume that $s(s_0-1)/2\in\bbN$.
Using Theorem \ref{thm_sun} for the elements $u',v'$ with Inequality \eqref{eq_II_1_top_ineq} we conclude that 
$$wu'w^*\in(v^G\cup v^{-G})^{24m \lceil (s_0-1) / (s/2)(s_0-1) \rceil}=(v^G\cup v^{-G})^{48m \lceil 1 / s \rceil}.$$
Hence 
$u\in (((v^G)_{\eps}\cup(v^{-G})_{\eps})^{48m \lceil 1 / s \rceil})_{\eps},$
where $G\coloneqq \cP\cU(\cM)$. Using Lemma \ref{lem_epsconj}, we obtain 
$$u\in ((v^G\cup v^{-G})^{48m \lceil 1 / s \rceil})_{(48m \lceil 1 / s \rceil+1)\eps}.$$
Now letting $\eps$ tend to zero, i.e., approximating both $u$ and $v$ finer and finer in the $2$-norm, we obtain 
$$u\in \overline{(v^G\cup v^{-G})^{48m \lceil 1 / s \rceil}}^{\norm{\cdot}_2}.$$
Summarizing the above discussion, we have proven the following theorem. 

\begin{thm}\label{thm_II_1}
 Let $\cM$ be a $\mathrm{II_1}$-factor. Assume that $u,v\in G\coloneqq\cP\cU(\cM)$ satisfy $\ell_{0}(u)\leq m\ell_t(v)$ for all $t\in [0,s]$ and some $m\in\bbN$. Then $$u\in \overline{(v^G\cup v^{-G})^{48m \lceil 1 / s \rceil}}^{\norm{\cdot}_2}.$$
Moreover, if both $u$ and $v$ have finite spectrum and rational spectral weights, then 
 $$u\in (v^G\cup v^{-G})^{48m \lceil 1 / s \rceil}.$$
\end{thm}

The preceding result already provides information on the normal generation function, at least when one restricts attention to elements with finite spectrum and rational spectral weights. Indeed, we can estimate how quickly we can generate a symmetry with trace zero and then use our improved version of Broise's theorem to generate the entire projective unitary group, i.e., Corollary \ref{broisecor}. In the next section we will show that any element can generate some comparable element with finite spectrum and rational spectral weights and then use this to finish the proof of our main result.

\section{Bounded normal generation for factors of type II$_{1}$}
\label{sec_alg_II_1}

The following result is a first observation on the spectral behaviour under taking appropriate commutators. 
\begin{lem}\label{lem_large_commutator}
 For every $u\in\cU(\cM)$  there exists $v\in\cU(\cM)$ such that 
 $$2 \cdot \norm{1-uvu^*v^*}_2 \geq \inf_{\lambda\in{S^1}}\norm{1-\lambda u}_2.$$
\end{lem}
\begin{proof}
 Apply \cite[Lemma 2.3]{P-81} (see also \cite[Lemma XIV.5.6]{Tak}) with $\eps= 1$ to the element $u-\tau(u)$ in order to obtain the existence of $v\in\cU(\cM)$ such that 
 \begin{align*}
\norm{v-uvu^*}_2^2=  \norm{v(u-\tau(u))v^*-(u-\tau(u))}_2^2\geq \norm{u-\tau(u)}_2^2
  \geq \frac14 \inf_{\lambda\in{S^1}}\norm{1-\lambda u}_2^2,
 \end{align*}
as claimed. In the last inequality, we used the (geometrically obvious) inequality $2 |z-z_0| \geq |z-z_0/|z_0||$ for all $z \in S^1$ and $z_0$ with $0<|z_0| \leq 1$.
\end{proof}

However, Lemma \ref{lem_large_commutator} does not reveal information about the generalized projective $s$-numbers of the commutator. It is much harder to keep track of that information under commutators. We now construct for a given unitary $u$ another unitary $v$ such that the commutator $[u,v]$ retains much of the spectral information of $u$. 
On the one hand this result is crucial for our proof of property {\rm (BNG)} in the $\mathrm{II_1}$-factor case, on the other hand it is of independent interest since it allows to consider commutators instead of the original element without qualitatively changing the (projective) spectral information.

\begin{prop}\label{prop_commutator}
 Let $\cM$ be a $\mathrm{II_1}$-factor. For every $u\in\cU(\cM)$ there exist $v\in\cU(\cM)$ and $\lambda \in S^1$ such that 
 $$\mu_{12t}(1-\lambda u)\leq 4\ell_t([u,v]) \textnormal{ for all } t\geq 0.$$
Moreover, $v$ can be chosen such that it has finite spectrum and rational spectral weights. 
\end{prop}

We will first prove a corresponding statement about matrix algebras and then use approximation to prove Proposition \ref{prop_commutator}.

\begin{lem} \label{auxlem}
Let $n \in \mathbb N$ and $u \in \cU(n)$. Then, there exists $v \in \cU(3n)$ and $\lambda \in S^1$ such that
$$\mu_i(1-\lambda u) \leq \sqrt{2}  \ell_i([\diag(u,u,u),v]), \quad \forall 0 \leq i \leq n-2.$$
\end{lem}
\begin{proof}
Let $u=\diag(\lambda_0,\dots,\lambda_{n-1})$ be optimal with associated permutation $\pi$. Note that Lemma \ref{lem_singularcage} gives $\mu_i(1-\lambda u) \leq |\lambda_{\pi(i)}-\lambda_{\pi(i)+1}|$ for some fixed $\lambda \in S^1$.
Let $\sigma$ be the standard cyclic permutation on the set $\{0,1,\dots,n-1\}$, i.e., $\sigma(i) = i+1 \mod n$. We set $v:= \diag(\sigma,\sigma^{-1},1_n)$.
The commutator $[\diag(u,u,u),v]$ has eigenvalues
$\lambda_{\pi(i)} \bar \lambda_{\pi(i)+1}$ and $\bar\lambda_{\pi(i)}  \lambda_{\pi(i)+1}$ for $i = 0,\dots,n-2$, $\lambda_n \bar \lambda_1, \bar \lambda_n \lambda_1$, and in addition $n$ eigenvalues equal to $1$.

We claim that each $z \in S^1$ and $0 \leq i \leq n-2$, at least $i+1$ of these eigenvalues are at distance at least $\frac1{\sqrt{2}}\mu_i(1-\lambda u)$ from $z$. Indeed, if $|z-1|\geq \sqrt{2}$, this is obvious since $\frac1{\sqrt{2}}\mu_i(1-\lambda u) \leq \sqrt{2}$ for all $i$. Now, if $|1-z|< \sqrt{2}$ then for each $0 \leq j \leq i$, we have
$$ {\sqrt{2}}\max\{ |z - \lambda_{\pi(j)} \bar \lambda_{\pi(j)+1}|, |z -  \bar\lambda_{\pi(j)} \lambda_{\pi(j)+1}| \} \geq |1 - \lambda_{\pi(j)} \bar\lambda_{\pi(j)+1}| \geq   |\lambda_{\pi(i)}-\lambda_{\pi(i)+1}|.$$ This finishes the proof.
\end{proof}
We are now ready to prove the corresponding statement for II${}_1$-factors.
\begin{proof}[Proof of Proposition \ref{prop_commutator}]
If $u$ is central, then the claim is trivial. So assume that $u$ is non-central. 
Let $s$ denote the projective rank of $u$. Let $\lambda\in S^1$ be the complex number satisfying $\mu_s(1-\lambda u)=0$.
For $\delta \coloneqq \ell_{s/2}(u)>0$ we obtain
 $\ell_{t}(u)\leq 2\ell_t(u)-\delta$ for all $t\in[0,s/2).$
 Using the right continuity of $\mu_t$ in $t$ we get the existence of $\delta_0>0$ such that 
 $\mu_{0}(1-\lambda u)-\mu_{12\delta_0}(1-\lambda u)\leq \delta/2$
 and thus
 \begin{align}
  \label{eq_delta0}
  \mu_{12t}(1-\lambda u)-\mu_{12\delta_0}(1-\lambda u)\leq \frac{\delta}{2} \textnormal{ for all } t\in[0,\delta_0).
 \end{align}
 
 Let $\eps>0$ such that $\eps \leq \delta\delta_0/40$.
 By Proposition \ref{prop_WvNBVD} we can find $u' \in \cU(\cM)$ and $n \in \mathbb N$ such that $\norm{u-u'}_2<\eps$ and $u'=\sum_{i=0}^{n-1} \lambda_i p_i$ with orthogonal projections $p_i$ and $\tau(p_i)=1/n$ for $i=0,\ldots, n-1$. Relabelling if necessary, we may assume that $\diag(\lambda_0,\ldots,\lambda_{n-1})$ is optimal with associated permutation $\pi$ -- see Definition \ref{def_optimal}. Note that we can choose $u'$ such that $\lambda_{n-1}=\overline{\lambda}$.
Applying Lemma \ref{auxlem} to $u'$, we obtain a unitary $v \in \cU(\cM)$ and $\lambda\in S^1$ such that
\begin{align}\label{eq_commutator}
  \mu_{t}(1-\lambda u')\leq \sqrt{2} \ell_{t/3}([u',v]) \textnormal{ for all } t\geq 0.
\end{align} 
Note that $ \norm{[u,v]-[u',v]}_2 < 2\eps.$
 Using Lemma \ref{lem_prop_s_numbers} we have the following estimates for every $t>0$:
 \begin{align*}
  \mu_t(1-\lambda u)&=\mu_t(1-\lambda(u-u'+u')) \leq \mu_{t/2}(1-\lambda u')+\mu_{t/2}(u-u').
 \end{align*}
 From Lemma \ref{lem_ineq} we conclude 
 \begin{align}
 \mu_t(1-\lambda u)&\stackrel{\hphantom{\eqref{eq_commutator}}}{\leq} \mu_{t/2}(1-\lambda u')+2\eps / t \stackrel{\eqref{eq_commutator}}{\leq} \sqrt{2}\ell_{t/6}([u',v])+2\eps / t. \label{eq_u_uv}
 \end{align}
 The same calculation with $u$ replaced by $[u',v]$ and $u'$ replaced by $[u,v]$ shows that
 \begin{align}
  \ell_t([u',v])\leq \ell_{t/2}([u,v])+2\cdot 2\eps / t.\label{eq_uv}
 \end{align}
 Combining Inequalities \eqref{eq_u_uv} and \eqref{eq_uv} we get
 \begin{align}\label{eq_combined}
 \mu_t(1-\lambda u)\leq \sqrt{2}\ell_{t/12}([u,v])+\min\set{10\eps / t,2}\quad\mbox { for all }t\geq 0.
 \end{align} 
 From the inequality $\mu_{t}(1-\lambda u)\leq 2\mu_t(1-\lambda u)-\delta \textnormal{ for all } t\in[0,s/2)$ and the above estimates we conclude for $t\in[0,s/12)$ that
 \begin{align*}
  \mu_{12t}(1-\lambda u) &\stackrel{\hphantom{\eqref{eq_uv}}}{\leq} 2\mu_{12t}(1-\lambda u)-\delta 
  \stackrel{\eqref{eq_combined}}{\leq} 4\ell_{t}([u,v])+ \frac{20\eps}{t} - \delta.
 \end{align*}
 Using Equation \eqref{eq_delta0} we obtain from the above inequality that for all $t\in[0,\delta_0)$
 \begin{align*}
  \mu_{12t}(1-\lambda u)&\leq \mu_{12\delta_0}(1-\lambda u) +\frac{\delta}{2}
  \leq 4\ell_{\delta_0}([u,v]) +\frac{20\eps}{\delta_0} - \delta  +\frac{\delta}{2} 
  \leq 4\ell_t([u,v]) + \frac{20\eps}{\delta_0} - \frac{\delta}{2}\\
  &\leq 4\ell_t([u,v]).
 \end{align*}
 If $s/12\geq t\geq\delta_0$ we have
 \begin{align*}
  \mu_{12t}(1-\lambda u)&\leq 4\ell_{t}([u,v])+ \frac{20\eps}{t} - \delta 
  \leq 4\ell_{t}([u,v])+ \frac{20\eps}{\delta_0} - \delta 
  \leq 4\ell_t([u,v]).
 \end{align*}
 Since $\mu_{12t}(1-\lambda u)=0$ for all $t\geq s/12$ we can summarize our estimates to
 $\mu_{12t}(1-\lambda u)\leq 4\ell_t([u,v])$ for all $t\geq 0,$
 which concludes the proof.
\end{proof}

We will also need the following Borel measurable version of Lemma \ref{lem_su2}.

\begin{lem}\label{lem_Borel_su2}
 Let $(X,\nu)$ be a Borel measure space and let $u=\left(\begin{smallmatrix}e^{\complex\varphi} &0\\ 0 &e^{-\complex\varphi}\end{smallmatrix}\right),\ v= \left(\begin{smallmatrix}e^{\complex\theta} &0\\ 0 &e^{-\complex\theta}\end{smallmatrix}\right)\in G\coloneqq \cU( M_{2\times 2}(\bbC)\otimes L^{\infty}(X,\nu))$ be non-trivial elements. If $\abs{\varphi(x)}\leq m\abs{\theta(x)}$ for some even $m\in\bbN$ and $\nu$-almost every $x\in X$, then $u\in(v^G)^m$.
\end{lem}
\begin{proof}
 The proof follows closely the proof of Lemma \ref{lem_su2} that can be found in \cite{NS-12}, but we need to ensure that the steps are Borel. This will be clear from the construction.
 
 Observe that multiplication of diagonal elements by $v(x)$ adds the angle $\theta(x)$ while multiplication with $v^{-1}(x)$ subtracts the angle $\theta(x)$. 
 If $w(x)\in\cSU(2)$ is diagonal with $\tr(w(x))=\cos\gamma(x)$, then we can choose $v'(x)\in v(x)^{\cSU(2)}$ such that $\tr(w(x)v'(x))=\cos\gamma_1(x)$ for any $\gamma_1(x)\in[\gamma(x)-\theta(x),\gamma(x)+\theta(x)]$, namely 
 $$v'(x)\coloneqq\left(\begin{smallmatrix} \cos\theta(x)+\complex\sin\theta_1(x) &b(x)\\ -\bar{b}(x) &\cos\theta(x)-\complex\sin\theta_1(x) \end{smallmatrix}\right)$$
 for $\theta_1(x)\in[-\theta(x),\theta(x)]$, where $\abs{b(x)}^2=1-\cos^2\theta(x)-\sin^2\theta_1(x)=\sin^2\theta(x)-\sin^2\theta_1(x)\geq 0$. 
 Assume that $\varphi(x)$ and $\theta(x)$ have the same sign for $\nu$-almost all $x\in X$ (else one needs to replace $v$ by $v^*$ in the following). 
 Multiply $v(x)$ $n\in\set{1,\ldots,m-1}$ times by itself until either $\varphi(x)\leq n\theta(x)$ or $\varphi(x)\geq (m-1)\theta(x)$. In the second case, multiplying $v^{m-1}(x)$ by the right element $v'(x)$ one obtains $u(x)=v^{m-1}(x)\cdot v'(x)$. In the first case, if $n=m-1$ then we also get $u(x)=v^{m-1}(x)\cdot v'(x)$. If $n<m-1$ then we multiply interchangingly by $v^*(x)$ and $v(x)$ until one step is left. The last step is to use the conjugate $v'(x)$ of $v(x)$ to obtain $u(x)=v^{n}(x)v^*(x)v(x)\cdot\ldots\cdot v^*(x) v(x) \cdot v'(x)$. 
 This gives an algorithm which terminates in finitely many steps and divides $X$ into Borel sets in each step.
\end{proof}

Before proving the main result of this section we want to outline the strategy of the proof. The aim is to generate an element $u\in\cP\cU(\cM)$ having finite spectrum and rational weights with an arbitrary element $v\in\cP\cU(\cM)$ under the assumption of an inequality between their projective $s$-numbers. Our first step is to map $v$ via an isomorphism into $2\times 2$ matrices over $p\cM p$ with $\tau(p)=1/2,$ such that they have diagonal form. Then $v=\left(\begin{smallmatrix} v_0 & 0 \\ 0 & v_1 \end{smallmatrix}\right)=\left(\begin{smallmatrix} v_0 & 0 \\ 0 & 1 \end{smallmatrix}\right) \cdot \left(\begin{smallmatrix} 1 & 0 \\ 0 & v_1 \end{smallmatrix}\right)$. Using Proposition \ref{prop_commutator} we can ensure that the projective singular values of $[v_0,w_0]$, where $w=\left(\begin{smallmatrix} w_0 & 0 \\ 0 & 1 \end{smallmatrix}\right)$, are still comparable with those of the original element $v$. We then use two conjugates of $[v,w]g[v,w]^{-1}g^{-1}$ to construct (using Lemma \ref{lem_Borel_su2}) a unitary $v'$ which has finite spectrum and rational spectral weights, where $g$ permutes the diagonal entries of the $2\times 2$ matrix $[v,w]$. Using now Theorem \ref{thm_II_1} we can generate $u$ with $v'$.

\begin{thm}\label{thm_main}
 Let $\cM$ be a separable $\mathrm{II_1}$-factor and $u,v\in G\coloneqq\cP\cU(\cM)$. Assume that $u$ has finite spectrum and rational spectral weights. If $\ell_{0}(u)\leq m\ell_t(v)$ for all $t\in [0,s]$ and some $m\in\bbN$, then
 $$u\in (v^G\cup v^{-G})^{18432 m\lceil 1/s\rceil}.$$
\end{thm}

\begin{proof} We lift $u$ and $v$ to $\cU(\cM)$ such that $\ell_0(u)\leq m \ell_t(v)$ and argue in $\cM$.
 First note that for $\delta\coloneqq \ell_{s}(v)>0$ we have 
$$\ell_0(u)\leq m(2\ell_t(v)-\delta) \quad\mbox{ for all } t\in[0,s].$$
 Put $\eps\coloneqq \delta /4$. 
 Let $\zeta_1,\ldots,\zeta_n$ be $n$ roots of unity with $\arg(\zeta_i)<\arg(\zeta_{i+1})$ such that for every $\lambda\in\sigma(v)$ there is an $i\in\{1,\ldots,n\}$ such that $\abs{\lambda-\zeta_i}<\eps$. We may assume that there exists no $\zeta_i$ satisfying $\abs{\lambda-\zeta_i}>\eps$ for all $\lambda\in\sigma(v)$. Denote by $p_i$ the spectral projection of $v$ corresponding to the set $\set{e^{\complex\varphi}\mid \varphi\in[\arg(\xi_i),\arg(\xi_{i+1}))}$, where $\zeta_{n+1}\coloneqq\zeta_1$ and $i\in\set{1,\ldots,n}$. Define $f(v)=\sum_{i=1}^{n'}\zeta_ip_i$.
 It follows that 
$\norm{v-f(v)}<\eps.$
 Now take subprojections $p_i'$ of $p_i$ with $\tau(p_i')=\frac{1}{2}\tau(p_i)$. 
Let $p\coloneqq \sum_{i=0}^{n'}p_i'$. Then $\tau(p)=1/2$ and $p$ commutes with $v$.

Denote in the following by $\ell_t^{(p)}$ the restriction of $\ell_t$ to $p\cM p$, i.e., $\ell_t^{(p)}(x)=\inf_{\lambda}\mu_t(p-\lambda pxp)$ for $x\in\cM$. We conclude 
that
$\ell_{2t}(v)\leq \ell_{2t}(f(v))+ \eps  =\ell_t^{(p)}(f(v)) +\eps$  for every $t\geq 0.$
Since we also have $\norm{f(v)p-vp}<\eps$ we obtain $\ell_{t}^{(p)}(f(v))\leq \ell_t^{(p)}(v) + \eps$ for all $t\geq 0$ and thus
\begin{align}\label{eq_3}
  \ell_{2t}(v)\leq \ell_t^{(p)}(v) + 2\eps \mbox{ for every }t\geq 0.
\end{align}

We have $v\cong \left(\begin{smallmatrix} v_0 & 0 \\ 0 & v_1 \end{smallmatrix}\right)=\left(\begin{smallmatrix}v_0 &0\\ 0 &1\end{smallmatrix}\right)\cdot \left(\begin{smallmatrix}1 &0\\ 0 &v_1\end{smallmatrix}\right) \in \cU(p\cM p\otimes M_{2\times 2}(\bbC))$ for $v_0\coloneqq vp$ and some $v_1\in \cU(p\cM p)$.
By Proposition \ref{prop_commutator} applied to the algebra $p\cM p$ there exists $w=\left(\begin{smallmatrix}w_0 &0\\ 0 &1\end{smallmatrix}\right)\in\cU(p\cM p\otimes M_{2\times 2}(\bbC))$ such that
$\ell_{12t}^{(p)}(v)\leq 4\ell_{t}^{(p)}([v,w])$ for all $t\geq 0,$
where $[v,w]=\left(\begin{smallmatrix}v_0w_0v_0^*w_0^* &0\\ 0&1\end{smallmatrix}\right).$
Let $g\in \cU(p\cM p\otimes M_{2\times 2}(\bbC))$ be such that 
$g[v,w]^{-1}g^{-1}=\left(\begin{smallmatrix}1 & 0 \\ 0 & (v_0w_0v_0^*w_0^*)^{-1}\end{smallmatrix}\right).$
Then under the identification of $G$ with its image under the isomorphism $\cM\rightarrow p\cM p\otimes M_{2\times 2}(\bbC)$ we have
${v''}\coloneqq [v,w]g[v,w]^{-1}g^{-1}\in (v^G\cup v^{-G})^4.$
In particular,
\begin{align}\label{eq_til_v}
\ell_{12t}^{(p)}(v)\leq 4\ell_{t}^{(p)}([v,w])= 4\ell_t^{(p)}({v''}) \textnormal{ for all } t\geq 0.
\end{align}

By Theorem II.6.1 in \cite{Dix-81} we can decompose $L^{\infty}(\sigma({v''}),\nu)$ into a direct integral such that ${v''}$ is represented as $\int_{\sigma({v''})}^{\oplus} \left(\begin{smallmatrix} \lambda &0\\ 0& \overline{\lambda}\end{smallmatrix}\right)d\nu(\lambda)$.

Now we can use Lemma \ref{lem_Borel_su2} to generate an element $v'$ with discrete spectrum and rational spectral weights such that $\ell_{t}(v')+\eps \geq \ell_t({v''})$ for all $t\geq 0$ and $$v'\in({v''}^G\cup {v''}^{-G})^2\subseteq (v^G\cup v^{-G})^8.$$
In the following, we describe how to generate such an element explicitly. First note that right continuity of $\ell_t$ in $t\geq 0$ implies that there exists $\delta_0\in(0,1)$ such that
 $\ell_0({v''})\leq \ell_{\delta_0}({v''}) + \eps.$ Let $\eps_0\in (0,1)$ be such that $\eps_0<\delta_0\delta /24$ and such that there exists $\lambda\in\sigma({v''})$ with $\abs{1-\lambda}>\eps_0$. 
 Let $\lambda_0\coloneqq 1$ and $\lambda_1,\ldots,\lambda_n\in{S^1}$, $n\in\bbN$, such that $\abs{1-\lambda_i}\geq \eps_0$ for all $i\in\set{1,\ldots,n}$ and $\abs{1-\lambda_i}= \eps_0$ for $i=1,n$,
 for every $\lambda\in\sigma(v'')$ with $\abs{1-\lambda}\geq \eps_0$ there exists $i\in\set{1,\ldots,n}$ such that 
  $\abs{\lambda-\lambda_i}<\eps_0,$
 $\abs{\varphi_i} < \abs{\varphi_{i+1}}\leq 2\abs{\varphi_{i}}$ for all $i\in\set{1,\ldots, n-1}$, where $\varphi_i=\arg(\lambda_i)\in[0,2\pi)$.
 Denote the subprojections of the spectral projections of ${v''}$ corresponding to the parts 
 $$(\varphi_1/2,\varphi_1],(\varphi_1,\varphi_2],\ (\varphi_2,\varphi_3],\ldots, (\varphi_{n-1},\varphi_n],\ (\varphi_n,\varphi_1/2]$$
 by $p_1,\ldots, p_n, p_0$. Then $\sum_{i=0}^np_i=1$. Without loss of generality all these projections are non-trivial (else we can leave out some parts and renumber). Let $q_i\precsim p_i$ for $i=0,\ldots,n$ be subprojections of rational trace such that $\tau(p_i-q_i)<\eps_0 / n$.
 \item Using Lemma \ref{lem_Borel_su2} we can generate 
$v'=\sum_{i=0}^n\lambda_i q_i + q'$
 in two steps, where $q'\coloneqq 1-\sum_{i=0}^nq_i$, $\tau(q')\leq 1-(1-n\cdot \eps_0/n)=\eps_0$. 

We have generated a unitary with finite spectrum and rational spectral weights.
This allows us to conclude
 \begin{align*}
  \norm{v'-{v''}}_1 &\leq \sum_{i=0}^n\norm{(v'-{v''})q_i}_1 + \norm{(v'-{v''})q'}_1 \\
  &\leq \sum_{i=0}^n\norm{(v'-{v''})q_i}\cdot \norm{q_i}_1 + \norm{v'-{v''}}\cdot \norm{q'}_1 \\
  &< \eps_0 \cdot \sum_{i=0}^n \norm{q_i}_1 + 2\cdot \norm{q'}_1
  < \eps_0 +2\eps_0 = 3\eps_0.
 \end{align*}
 Thus for $t\in[0,\delta_0/2)$ we conclude
 \begin{align*}
  \ell_{2t}({v''})&\leq \ell_{\delta_0}({v''})+\delta/4
  \leq \ell_{\delta_0/2}(v') + \frac{6\eps_0}{\delta_0} + \delta/4 
  < \ell_{\delta_0/2}(v')+\delta/2.
 \end{align*}
 For $t\geq \delta_0 /2$ we obtain
$ \ell_{2t}({v''}) \leq \ell_t(v') + \frac{6\eps_0}{\delta_0} \leq \ell_t(v') + \delta /4,$
 so that we have 
 \begin{align}\label{eq_[v,w]_v'}
  \ell_{2t}({v''})\leq \ell_t(v') + \delta /2 \mbox{ for all } t\geq 0,
 \end{align}
 as well as 
$\ell_{2t}^{(p)}({v''})\leq \ell_t^{(p)}(v') + \delta /2 \mbox{ for all } t\geq 0.$

From $\ell_{0}(u)\leq m(2\ell_t(v)-\delta)$ for all $t\in [0,s]$ and from Equation \eqref{eq_[v,w]_v'} we conclude for all $t\in [0,s]$ that 
\begin{align*}
 \ell_{0}(u) &\stackrel{\hphantom{\eqref{eq_3}}}{\leq} m(2\ell_{t}(v)-\delta)
\stackrel{\eqref{eq_3}}{\leq} m(2\ell_{t/2}^{(p)}(v) + 2\eps -\delta) \stackrel{\eqref{eq_til_v}}{\leq} m(8\ell_{t/24}^{(p)}({v''}) +2\eps -\delta) \\
 &\stackrel{\eqref{eq_[v,w]_v'}}{\leq} m(8\ell_{t/48}^{(p)}(v') + \delta /2 +2\eps -\delta)\stackrel{\hphantom{\eqref{eq_3}}}{\leq} 8m\ell_{t/48}^{(p)}(v')\stackrel{\hphantom{\eqref{eq_3}}}{\leq} 8m\ell_{t/48}(v').
\end{align*}
Summarizing these estimates we have
\begin{align}\label{eq_important}
 \ell_{0}(u)\leq 8m\ell_{t}^{(p)}(v') \leq 8m\ell_{t}(v') \quad\mbox{ for all } t\in [0,s/48].
\end{align}

Since $u$ has finite spectrum and rational weights we can use Theorem \ref{thm_II_1} to obtain:
\begin{align*}
 u\in((v')^G \cup (v')^{-G})^{48m\lceil 48/s\rceil} &\subseteq((v')^G \cup (v')^{-G})^{2304m\lceil 1/s\rceil} \subseteq (v^G\cup v^{-G})^{18432m\lceil 1/s\rceil}.
\end{align*}
This concludes the proof.
\end{proof}

In Theorem \ref{thm_main} we required the element $u$ to have finite spectrum and rational spectral weights. So in particular, we can generate any symmetry of trace $0$. To prove that $\cP\cU(\cM)$ has property {\rm (BNG)} it then suffices then to combine Theorem \ref{broise_improved} (respectively Corollary \ref{broisecor}) and Theorem \ref{thm_main}.

\begin{thm}\label{thm_II_1_BNG}
 The projective unitary group of a separable $\mathrm{II_1}$-factor has property {\rm (BNG)}.
\end{thm}
\begin{proof}
 Let $v\in G\coloneqq \cP\cU(\cM)\setminus\set{1}$ be arbitrary and denote by $s$ its projective rank. Let $w$ be a symmetry of trace $0$. 
By Lemma \ref{lem_right_cts} there exists $\eps>0$ such that 
$\ell_t(v)\geq\eps$ for all $t\in[0,s/2].$ We obtain
$\ell_{0}(w)\leq \lceil 2 / \eps \rceil \ell_t(v)$ for all $t\in[0,s/2].$
Using Theorem \ref{thm_main} we obtain
 $w\in (v^G\cup v^{-G})^{18432m\lceil 1/s\rceil}.$ Using now Corollary \ref{broisecor} we obtain 
$u\in (w^G\cup w^{-G})^{32}$
for any $u\in G$. That is, 
$$G=(v^G\cup v^{-G})^{589824 m\lceil 1/s\rceil}.$$
This finishes the proof.
\end{proof}

Theorem \ref{thm_II_1_BNG} easily implies the algebraic simplicity of $\cP\cU(\cM)$ which was first discovered by de la Harpe - see the main theorem in \cite{dlH-79}.

\begin{cor}
 The projective unitary group of a $\mathrm{II_1}$-factor is simple.
\end{cor}

Consider the natural length function 
$$\ell(u)=\inf_{\lambda\in{S^1}}\norm{1-\lambda u}_1$$
and consider also $L(u) = \int_0^1 \ell_t(u) dt$.
Note that
$$\ell(u) = \inf_{\lambda\in{S^1}}\norm{1-\lambda u}_1 = \inf_{\lambda\in{S^1}} \int_{0}^1 \mu_t(1-\lambda u) dt \geq  \int_{0}^1 \inf_{\lambda\in{S^1}} \mu_t(1-\lambda u) dt = L(u).$$

\begin{lem} \label{Llbound}
There exists a constant $c>0$ such that $\ell(u) \leq c \cdot L(u)$.
\end{lem}
\begin{proof}
By Proposition \ref{prop_commutator}, there exist $v \in \cU(\cM)$ and $\lambda \in S^1$ such that
$\mu_{12t}(1-\lambda u)\leq 4\ell_t([u,v])$ for all $t \geq 0$. Since $\ell_t(uvu^*v^*) \leq 2\ell_{t/2}(u)$, we obtain
$$\ell(u) \leq \int_0^1 \mu_{t}(1-\lambda u) dt \leq 4\int_{0}^1 \ell_{t/12}([u,v]) dt \leq 8 \int_{0}^1 \ell_{t/24}(u) dt \leq 192 \cdot L(u).$$ This proves the claim.
\end{proof}

\begin{cor}
There exists a universal constant $c$ such that the following holds. Let $G$ denote the projective unitary group of a separable $\mathrm{II_1}$-factor and assume that $v\in G\setminus\set{1}$. Then
$G=\left( v^G\cup v^{-G}  \right)^{k}$
for every $k\geq c|\log \ell(v) | / \ell(v)$.
\end{cor}
\begin{proof}
Observe that $t \mapsto \ell_{t}(v)/2$ is a non-zero and non-increasing self-map of $[0,1]$. We set $L:= \int_{0}^1 \ell_t(v)/2 \ dt.$ Assume for a moment that $L \leq 1/3$. From \cite[Lemma 2]{Th-14} we conclude that there exists some $t_0\in[0,1]$ such that 
 $$t_0\ell_{t_0}(v)\geq \frac{L}{-2\log(L)}.$$
 As in the proof of Theorem \ref{thm_II_1_BNG} we conclude that 
 $$G=(v^G\cup v^{-G})^{c\cdot \lceil 1/\ell_{t_0}(v)\rceil \cdot \lceil 1/ t_0\rceil}\subseteq (v^G\cup v^{-G})^{k},$$
 for any $k \geq  c |\log(\ell(v))|/\ell(v) \geq 2c |\log(L)|/ L$.
Now if $L> 1/3$, then we have $\ell_{1/6}(v)/2\geq 1/6$. Indeed, assume to the contrary that $\ell_{1/6}(v)/2<1/6$, then we would have
 $$L\leq \int_{[0,1/6]}1dt + \int_{[1/6,1]}\frac{1}{6} dt \leq \frac{1}{6}+ \frac{1}{6}=\frac{1}{3},$$
 a contradiction to $L>1/3$. Thus, we will be able to quickly generate $\cP\cU(\cM)$ in this case. Possibly enlarging $c$, we obtain (using Lemma \ref{Llbound}) that the function $f \colon G\setminus\set{1}\rightarrow \bbR$ defined by
$f(v)\coloneqq c |\log(\ell(v))|/\ell(v)$ is a normal generation function.
\end{proof} 

\begin{rem}
Note that the normal generation function is only off by a logarithmic factor from the obvious lower bound, given by inverse of the length function itself, see Proposition \ref{prop_lower_bound}.
\end{rem}

\section{Automatic Continuity and Uniqueness of the Polish group topology}
\label{cha_autocont}

Automatic continuity properties of groups of functional analytic type is a classical subject, see for example 
\cite{BK-96,BYBM-13,Dud-61,Pe-50,RS-07,Sa-13,Sak-60,Slu-13,Tkv-13}.

The aim of this section is to prove that every homomorphism from the group $\cP\cU(n)$, $n\in\bbN$, endowed with the norm topology, or $\cP\cU(\cM)$, $\cM$ a separable $\mathrm{II_1}$-factor, endowed with the strong operator topology,
into any separable SIN group is continuous. Recall, a polish group is called SIN (small invariant neighborhoods) if it has a basis of conjugation-invariant neighborhoods of the identity.
In general we say that a Polish group $G$ has automatic continuity if every homomorphism of $G$ into any other separable topological group is continuous. It is known that $\cP\cU(n)$ does not have automatic continuity -- for example matrix groups such as $\mathrm{SO}(3,\mathbb{R}) = \cP\cU(2)$ embed discontinuously into the group $\mathrm{S}_{\infty}$ of all permutations on $\mathbb{N}$ (this is \cite[Example 1.5]{Ros-09}, it follows from results of R. R. Kallman \cite{Ka-00} and S. Thomas \cite{Th-99}). For information on automatic continuity consult Rosendal's excellent survey \cite{Ros-09}.

Another goal is to prove the uniqueness of the Polish group topology of the projective unitary group of a separable $\mathrm{II_1}$-factor. To the author's knowledge this was previously unknown even for the hyperfinite $\mathrm{II_1}$-factor.
Throughout this section $\mathrm{II_1}$-factors are assumed to be separable.

In \cite{RS-07} Rosendal and Solecki develop a general framework for groups having automatic continuity. 
\begin{defn}
A topological group $G$ is \textbf{Steinhaus}\index{topological group!Steinhaus} (\textbf{with exponent $k$}) if there exists an element $k\in\mathbb{N}$ such that $W^k$ contains an open neighbourhood of $1_G$ for any symmetric countably syndetic set $W\subseteq G$ (see Definition \ref{def_syndetic}).
\end{defn}
In Proposition 2 of \cite{RS-07} the authors can show the following. 
\begin{prop}[Rosendal-Solecki]\label{prop_RS}
 Every homomorphism from a Steinhaus topological group into any separable topological group is continuous. 
\end{prop} 
 For example, topological groups with ample generics are Steinhaus with exponent 10 (see \cite[Lemma 6.15]{KR-07}). Rosendal and Solecki show that the group $\mathrm{Aut}(\mathbb{Q},<)$ of order-preserving bijections of the rationals and several homeomorphism groups are Steinhaus. Their proofs crucially use the existence of comeager conjugacy classes (the group $\mathrm{Homeo}_+(S^1)$ of orientation preserving homeomorphisms on the unit circle $S^1$ only has meager conjugacy classes, but the proof heavily uses that the group $\mathrm{Homeo}_+(\mathbb{R})$ of increasing homeomorphisms of $\mathbb{R}$ is Steinhaus, which in turn relies on the existence of comeager conjugacy classes). 
 
Before heading towards our proof of invariant automatic continuity of $\cP\cU(\cM)$ we show that the conjugacy classes in $\cP\cU(\cM)$ are meager.

\begin{prop}
 All conjugacy classes in the (projective) unitary group $G$ of a $\mathrm{II_1}$-factor, endowed with the strong operator topology, are meager.
\end{prop}
\begin{proof}
 The trace property of $\tau$ implies $\tau(g^G)=\tau(g)$ for all $g\in G$. Moreover, we have $\tau(\overline{g^G})=\tau(g)$, i.e., $g^G$ is nowhere dense by the Baire category theorem.  
\end{proof}

 This indicates that we need some new ideas to show an automatic continuity result for unitary groups of $\mathrm{II_1}$-factors. 
Indeed, our strategy to prove automatic continuity of projective unitary groups of $\mathrm{II_1}$-factors (endowed with the strong operator topology) differs greatly from the ones used before. The main ingredients in our proof are Theorem \ref{thm_main} and Propositon \ref{prop_nbhd} which ensures that a fixed power of any conjugacy-invariant countably syndetic set contains a neighborhood of the identity. The rest of our proof is an adaption of Proposition \ref{prop_RS} (cf. \cite[Proposition 2]{RS-07}).\\

The work of Rosendal and Solecki in \cite{RS-07} shows that the right sets to concider in order to get an  abstract automatic continuity result are so-called countably syndetic sets.

\begin{defn}\label{def_syndetic}
Let $W$ be a subset of a group $G$. We say that $W$ is \tb{symmetric}\index{set!symmetric} if $W=W^{-1}$. A symmetric set $W$ is called \textbf{countably syndetic}\index{set!countably syndetic} if there exist countably many elements $g_n\in G,\ n\in\bbN,$ such that $G=\bigcup_{n\in\bbN}g_nW$.
\end{defn}
An example of a countably syndetic set in a separable topological group is any nonempty open symmetric set.

In a semi-finite von Neumann algebra $\cM$ with faithful semi-finite normal trace $\tau$, one can measure the size of the support of an element $x\in\cM$ as follows. We define
$$\left[x\right]\coloneqq\inf\set{\tau(p)\mid p\in\Proj(\cM),\ p^{\bot}x=0}.$$
We observe that $\left[x\right]$ equals the trace of the support projection $s=s(x)$ of $x$ and that $[x_1+x_2]\leq[x_1]+[x_2]$ by \cite[Lemma 2.1]{Th-08}, and hence
$d_r(x,y)\coloneqq\left[x-y\right]$
satisfies the triangle inequality and thus defines a metric on $\cM$. Following the convention from \cite[Section 2.1]{Th-08} we call $d_r$ the \textbf{rank metric}\index{rank metric}.

To ensure that every countably syndetic set in the projective unitary group of a $\mathrm{II_1}$-factor contains \emph{large} elements, we need the following standard facts. We use the notation 
$B_r^d(x)\coloneqq\set{y\in X\mid d(x,y)\leq r}$
for a metric space $(X,d)$ and $x\in X$.

\begin{prop}\label{prop_not_sep}
The (projective) unitary group of a II${}_1$-factor is not separable in both (i) the uniform topology and (ii) the topology induced by the rank metric.
\end{prop}
\begin{proof} (i) This is well-known. One can prove it directly or use that $\cM$ contains an inseparable abelian von Neumann algebra and then use that every element in $\cM$ is a linear combination of four unitaries in $\cM$ to conclude that the unitary group is also inseparable.\\ 
 (ii) Set $u_{\varphi}\coloneqq p+ e^{\complex \varphi}p^{\bot}$, where $p\in\Proj(\cM)$ satisfies $\tau(p)=1/2$ and $\varphi\in[0,\pi/4]$. Then $B_{1/4}^{d_r}(u_{\varphi})$ for $\varphi\in[0,\pi/4]$ defines an uncountable family of disjoint open sets in $\cU(\cM)$ as well as $\cP\cU(\cM)$. Hence $\cU(\cM)$ and $\cP\cU(\cM)$ are not separable in the topology induced by the metric $d_r$.
\end{proof}

Proposition \ref{prop_not_sep} will ensure that for every countably syndetic set $W$ in $\cP\cU(\cM)$, $W^2$ contains elements of some suitable length in the above two inseparable topologies. In order to prove this, we need the following elementary lemma. 

\begin{lem}\label{lem_insep}
 Suppose that $(X,d)$ is an inseparable metric space. Then there exists $\eps>0$ such that for every countable subset $A$ of $X$ there exists $x\in X$ with $d(x,A)\geq \eps$.
\end{lem}
\begin{proof}
 Suppose there exists no such $\eps$. Then there exists a sequence $\set{A_n}_{n\in\bbN}$ of countable 	subsets $X$ and a sequence $\set{\eps_n}_{n\in\bbN}$, $\eps_n\rightarrow 0$ for $n\rightarrow \infty$ such that for every $x\in X$ we have $\eps_{n}>d(x,A_n)$. But then $d(x,\bigcup_{n\in\bbN}A_n)=0$ for all $x\in X$. Thus $\bigcup_{n\in\bbN}A_n$ forms a countable dense set in $X$, which is a contradiction.
\end{proof}

\begin{prop}\label{prop_insep}
 Let $G$ be an inseparable topological group with compatible left-invariant metric $d$. There exists $\eps>0$ such that for every countably syndetic set $W\subseteq G$, $W^2$ contains an element $u$ satisfying $d(1,u)>\eps$.
\end{prop}
\begin{proof}
 For the moment, let $\eps>0$ be arbitrary. Recall that an $\eps$-separated set $V\subseteq G$ is a set such that every pair of distinct points $u,v\in V$ has distance $d(u,v)>\eps$. Zorn's lemma implies that there exists a maximal $\eps$-separated set $V$. Observe that $V$ is $\eps$-dense in $G$ by maximality, since the existence of a point $u\in G\setminus V$ such that $d(u,v)>\eps$ for all $v\in V$ obviously contradicts maximality of $V$. 

 We conclude from Lemma \ref{lem_insep} that there exists $\eps>0$ such that any maximal $\varepsilon$-separated set $V$ is uncountable. We may assume that $1\in V$. Since $W$ is countably syndetic, there exists a sequence $(g_n)_n$ in $G$ such that $G=\bigcup_{n\in\bbN}g_n W$. In particular we have $V=\bigcup_{n\in\bbN}V\cap g_nW$.
 The pigeonhole principle implies that there exists $m\in\bbN$ such that 
 $$\abs{V\cap g_m W}\geq 2.$$
 Let $u,v\in V\cap g_m W$ be distinct elements. Since $W^2=(g_m W)^{-1}(g_m W)$ we get $u^{-1}v \in W^2$. Since $V$ is $\varepsilon$-separated we get $d(1,u^{-1}v)=d(u,v)> \varepsilon$ and this completes the proof.
\end{proof}

Let us come to the main definition of this section.

\begin{defn}
 Let $G$ be a topological group. If every homomorphism from $G$ to any separable SIN group is continous, then we say that $G$ has the \textbf{invariant automatic continuity property}\index{invariant automatic continuity property} or \textbf{property (\rm IAC)}.
\end{defn}

Closely related to invariant automatic continuity we define an invariant version of the Steinhaus property.

\begin{defn}
 A topological group $G$ has the \textbf{invariant Steinhaus property}\index{invariant Steinhaus property} (\textbf{with exponent $k$}) if there exists an element $k\in\mathbb{N}$ such that $W^k$ contains an open neighbourhood of $1_G$ for any symmetric conjugacy-invariant countably syndetic set $W\subseteq G$.
\end{defn}

Following closely the proof of \cite[Proposition 2]{RS-07} we obtain the invariant automatic continuity for groups having the invariant Steinhaus property.

%
%

\begin{prop}\label{prop_auto_Steinhaus}
 Let $G$ be a topological group with the invariant Steinhaus property. Then $G$ has the invariant automatic continuity property.
\end{prop}
\begin{proof}
 Let $\pi:G\rightarrow H$ be a homomorphism into a separable SIN group $H$. Assume that $G$ has the invariant Steinhaus property with exponent $k$. Clearly, if $\pi$ is continuous at the neutral element $1_G$ of $G$, then $\pi$ is continuous at every point $g\in G$.
 Suppose that $U\subseteq H$ is an open neighbourhood of $1_H$. Since $H$ is SIN we can find a conjugacy-invariant symmetric open set $V$ satisfying $1_H\in V\subseteq V^{2k}\subseteq U\subseteq H$. By separability of $H$, $V$ covers $H$ by countably many translates $\{h_n V\}_{n\in\mathbb{N}}$. For each $n\in\mathbb{N}$ such that $h_n V\cap\pi(G)\neq\emptyset$, choose $g_n\in G$ such that $\pi(g_n)\in h_nV$. Thus $h_n V\subseteq\pi(g_n)V^{-1}V=\pi(g_n)V^2$ and $\pi(g_n)V^2$ cover $\pi(G)$. 

For fixed $g\in G$, choose $n\in\mathbb{N}$ such that $\pi(g)\in\pi(g_n)V^2$. Then $\pi(g_n^{-1}g)\in V^2$, thus $g_n^{-1}g\in\pi^{-1}(V^2)$ and hence $g_n \pi^{-1}(V^2)$ cover $G$. Moreover, since $H$ is SIN we obtain $xg_n^{-1}gx^{-1}\in\pi^{-1}(V^2)$ for every $x\in G$. It follows that $W:=\pi^{-1}(V^2)$ is symmetric, countably syndetic and conjugacy-invariant in $G$.

Since $G$ has the invariant Steinhaus property by assumption, $W^{k}$ contains an open neighborhood of the identity.
Hence, $\pi(W^{k})\subseteq V^{2k}\subseteq U$, and we obtain $1_G\in\textnormal{Int}(\pi^{-1}(U))$, that is, $\pi$ is continuous at $1_G$.
\end{proof}

Let us now verify the invariant Steinhaus property for finite-dimensional projective unitary groups.

\begin{prop}\label{prop_PUn_Steinhaus}
 The projective unitary group $\cP\cU(n)$, endowed with the norm topology, where $n\in\bbN,$ has the invariant Steinhaus property with exponent $48n$. 
\end{prop}
\begin{proof}
 The case $n=1$ is trivial and so we assume $n\geq 2$.
 Put $G\coloneqq\cP\cU(n)$ and let $W\subseteq G$ be a symmetric conjugacy-invariant countably syndetic set. There exists $v\in W^2$ such that $d_r(1,v)>\eps$. We set $\delta\coloneqq \ell_0(v)>0$ and use $v$ to generate a $\delta$-neighborhood of the identity in the operator norm. So consider an arbitrary element $u\in G$ satisfying $\ell_0(u)\leq \delta$. From Theorem \ref{thm_sun} we then conclude 
$u\in(v^G\cup v^{-G})^{24n}.$ Since $u\in B_{\delta}^{\norm{\cdot}}(1)$ was arbitrary, this shows that $G$ has the invariant Steinhaus property with exponent $48n$.
\end{proof}

\begin{rem}
 Basically the same proof as above shows that $\cSU(n)$ also has the invariant Steinhaus property. The additional obstruction coming with $\cSU(n)$ is that it has a non-trivial center. However, the center is finite and thus one can generate a small $\delta$-neighborhood of the identity with $\delta>0$ and $\delta<\min_{\lambda\in\cZ(\cSU(n))\setminus\set{1}}\norm{1-\lambda}$.
\end{rem}

Propositions \ref{prop_auto_Steinhaus} and \ref{prop_PUn_Steinhaus} together with the previous remark imply the following.

\begin{thm}\label{thm_auto_PUn}
 $\cP\cU(n)$ and $\cSU(n)$, endowed with the norm topology, where $n\in\bbN,$ have the invariant automatic continuity property.
\end{thm}

We want to stress that $\cP\cU(n)$ and $\cSU(n)$ do not have the automatic continuity property \cite[Example 1.5]{Ros-09}, i.e., there is need for an extra condition on the class of target groups (also it is not clear if SIN groups form the most general such class). 
Note that $\cU(n)$ does not have the invariant automatic continuity property, since $\cU(n)$ maps to $S^1$ continuously.

Now we come to the core in our proof of the invariant automatic continuity property of projective unitary groups of separable $\mathrm{II_1}$-factors. A major difficulty in the proof stems from the fact that we could prove Theorem \ref{thm_main} in this quantitative version only if the element that one wants to generate has finite spectrum and rational spectral weights. Many of the techniques and results developed in the previous sections are needed. 

\begin{prop}\label{prop_nbhd}
 The projective unitary group $\cP\cU(\cM)$ of a separable $\mathrm{II_1}$-factor $\cM$, endowed with the strong operator topology, has the invariant Steinhaus property.
\end{prop}
\begin{proof}
Let $W\subseteq G\coloneqq \cP\cU(\cM)$ be a symmetric conjugacy-invariant countably syndetic set. We have to show that there exists a fixed $k\in\bbN$ (independent of $W$) such that $W^k$ contains a neighborhood of the identity.
 By Proposition \ref{prop_not_sep} and Proposition \ref{prop_insep} there exist
$\eps>0$ (independent of $W$) and $u,v\in W^2$ with $\norm{1-\lambda u}>\eps$ for all $\lambda\in{S^1}$,
$\ell_t(v)\neq 0$ for all $t\in[0,\eps]$.
 By right continuity of $\ell_t$ in $t$, see Lemma \ref{lem_right_cts}, there exist $\delta>0$ such that 
$\ell_t(u)\geq\eps$ for all $t\in[0,\delta]$ and $\ell_t(v)\geq\delta$ for all $t\in[0,\eps].$ 

To generate a neighborhood of the identity in the strong operator topology we need several steps.
First we use $u$ and $v$ to generate elements $w$ with $\norm{1- w}_2\leq \delta^2/2$ which are of the form 
  \begin{align}\label{eq_form}
  \begin{pmatrix}
    1 & 0 & 0\\
    0 & w_0 & 0\\
    0 & 0 & w_0^*
   \end{pmatrix}
  \end{align}
   in some $\cU(p_0\cM p_0 \otimes M_{3\times 3}(\bbC)) /S^1$, where $\tau(p_0)=1/3$. 
   Let $p$ be a projection commuting with $w$ such that 
   $$\norm{1-p-wp^{\perp}}<\delta \quad \mbox{ and }\quad   \tau(p)=\delta.$$
 Decompose $w=w_1w_2$ with $w_1\coloneqq wp+p^{\bot}$ and $w_2\coloneqq p+p^{\bot}w$. Hence $\ell_0(w_1)\leq \frac{2}{\eps}\ell_t(u)$ for all $t\in [0,\delta]$. Using Theorem \ref{thm_main} we can generate a symmetry $s$ of trace $0$ in $\cU(p\cM p)$ with $u$, namely we obtain $s\in(u^G\cup u^{-G})^{c\lceil 1/\eps\rceil}$ for some universal constant $c\in\bbN$. Corollary \ref{broisecor} allows us to conclude that 
 $$w_1\in (s^G\cup s^{-G})^{32}\subseteq (u^G\cup u^{-G})^{32c \lceil 1/\eps\rceil}.$$
 It remains to generate $w_2$. By Lemma \ref{lem_ineq} we have $\ell_{0}(w_2)\leq 2\delta^2/2\delta=\delta\leq\ell_t(v)$ for all $t\in [0,\eps]$.
 Suitable approximation of $v$ in the operator norm, as in the beginning of the proof of Theorem \ref{thm_main}, allows us to find a projection $p\in\cM$ which commutes with $v$, is equivalent to $p_0$ and such that for $\eps'=\delta/8$ we have
 $$\ell_{3t}(v)\leq \ell_t^{(p)}(v)+2\eps'.$$
 Now view $v$ as a (diagonal) element of $\cU(p\cM p \otimes M_{3\times 3}(\bbC))$.
  Using Proposition \ref{prop_commutator} we can find an element $v'=\left(\begin{smallmatrix}
    1 & 0 & 0\\
    0 & v_0' & 0\\
    0 & 0 & 1
   \end{smallmatrix}\right)\in \cU(p\cM p \otimes M_{3\times 3}(\bbC))$ 
   such that $\ell_{12t}^{(p)}(v)\leq 4\ell_t^{(p)}([v,v'])$ for all $t\geq 0$. Let $g\in\cU(p\cM p \otimes M_{3\times 3}(\bbC))$ be a unitary permuting the second and third diagonal entry. Consider the element 
   $${v''}\coloneqq [v,v']g[v,v']^{-1}g^{-1}\in (v^G\cup v^{-G})^4,$$ 
 and observe that ${v''}$ satisfies 
 $$\ell_{12t}^{(p)}(v)\leq 4\ell_t^{(p)}([v,v'])\leq 4\ell_t^{(p)}({v''})\quad \mbox{ for all }t\geq 0.$$
 Thus we have 
 $$\ell_0(w_2)\leq\delta \leq \ell_t(v)\leq \ell_{t/3}^{(p)}(v)+2\eps' \leq 4\ell_{t/36}^{(p)}({v''})+\frac{\delta}{4} \quad \mbox{ for all } t\in[0,\eps].$$
  As in the proof of Theorem \ref{thm_main} (restricting our attention to the lower $2\times 2$ part) we generate an element $v''\in ({v''}^G\cup {v''}^{-G})^2\subseteq (v^G\cup v^{-G})^8$ that has finite spectrum and rational weights such that
 $$ \delta\leq 4\ell_{t/36}^{(p)}({v''})+\frac{\delta}{4}\leq 4\ell_{t/36}^{(p)}(v'')+\frac{3\delta}{4} \quad\mbox{ for all } t\in[0,\eps].$$ 
 In particular, $4\ell_t^{(p)}(v'')\geq \delta/4$ for all $t\in[0,\eps/36]$. Hence
 \begin{align}\label{eq_estimate}
  \ell_0(w_2)\leq 16\ell_t^{(p)}(v'')\quad \mbox{ for all } t\in[0,\eps/36].
 \end{align}
 We restrict our attention to the lower $2\times 2$ subalgebra $q\cM q$ in \eqref{eq_form} and pass to the direct integral $M_{2\times 2}(L^{\infty}(\sigma(qw_2)),\nu)$, where $qw_2=\int_{\lambda\in\sigma(qw_2 )}\left(\begin{smallmatrix} \lambda & 0 \\ 0 & \overline{\lambda} \end{smallmatrix}\right)d\nu(\lambda)$ (note that $q$ commutes with $w_2$). 
 Let $p'$ denote the projection that cuts $v''$ down to the lower $2\times 2$ part. This allows us to conjugate $p'v''$ into $M_{2\times 2}(L^{\infty}(\sigma(qw_2)),\nu)$. Recall that Corollary \ref{cor_length_angles} gives us a relation between the projective $s$-numbers and the angles of the eigenvalues (note that $v''$ has finite spectrum and for $w_2$ we only need the estimate for the $0$-th projective $s$-number since $\ell_t(\cdot)$ is decreasing in $t$). We apply Lemma \ref{lem_Borel_su2} with the relation \eqref{eq_estimate} to generate $q'w_2$ for a subprojection $q'\leq q$, $\tau(q')=\eps/36$ (and $1$'s everywhere else). Thus using Lemma \ref{lem_Borel_su2} on at most $\lceil 36/\eps\rceil$ parts (where relation \eqref{eq_estimate} holds) we obtain 
 $$w_2\in (v''^G\cup v''^{-G})^{4\cdot16\cdot \lceil 36/\eps\rceil}\subseteq (v^G\cup v^{-G})^{8\cdot 64\cdot\lceil 36/\eps\rceil},$$
 (the factor $4$ comes from Corollary \ref{cor_length_angles}).
 We conclude that 
$w=w_1w_2\in W^{c\lceil 1/\eps\rceil}$
 for some constant $c\in\bbN$ (which is independent of $\delta$). 
 
 Assume now that $w$ is such that $\norm{1- w}_2\leq \delta^2$ has finite spectrum and rational weights. This case follows in the same way as in the first step. Namely one decomposes $w=w_1w_2$ and generates $w_1$ with the element $u$ (which has uniformly big projective $s$-numbers) and $w_2$ with the element $v$ (which has uniformly many non-trivial projective $s$-numbers). This leads us again to $w\in W^{c\lceil 1/\eps\rceil}$ for some constant $c\in\bbN$ which is independent of $\delta$. 

Assume that $w\in \mathrm{B}_{\eps_0}^{\norm{\cdot}}(1)\subseteq\cU(\cM)$ for some $\eps_0\in(0,\delta^2)$ small enough such that using Theorem \ref{broise_improved} we can decompose $w$ into a product $w_1\cdot\ldots\cdot w_8$ of elements $w_i\in\cU(\cM)$ of the form \eqref{eq_form} satisfying
 $$\norm{1-w_i}_2<\delta, \ \quad i=1,\ldots,8.$$
 Note that also $\eps_0$ depends of $W$.
 Using the first step, we obtain 
 $$w=w_1\ldots w_8\in W^{8c\lceil 1/\eps\rceil}.$$
 Thus we can generate an $\eps_0$-neighborhood in the operator norm in $8c\lceil 1/\eps\rceil$ steps.

Now let $w\in\mathrm{B}_{\eps_0}^{\norm{\cdot}_2}(1)$ be arbitrary. Approximate $w$ by an element $w'$ with finite spectrum in the operator norm, such that $\norm{w-w'}=\norm{1-ww'^*}<\eps_0$. From the third step we conclude that $ww'^*\in W^{8c\lceil 1/\eps\rceil}$. It remains to show that $w'$ can be generated from elements in $W^{C\lceil 1/\eps\rceil}$ for some constant $C\in\bbN$. Therefore, using Proposition \ref{prop_WvNBVD}, we approximate $w'$ with an element $w''$ that has finite spectrum and rational spectral weights such that 
$\norm{w'-w''}_2\leq \eps_1$ and $d_r(1,w'w''^*)\leq \delta.$
 The second step allows us to conclude $w''\in W^{c\lceil 1/\eps \rceil}$ for some constant $c\in\bbN$. We only have to generate the element $w'w''^*$ of small rank. It is clear that $\ell_t(w'w''^*)=0$ for all $t> \delta$. Let $q$ denote the projection witnessing non-triviality of $w'w''^*$ and observe that $\tau(q)\leq \delta$. As in the first step, we use $u$ to generate a symmetry $s$ of trace $0$ in $q\cM q$ such that
$s\in W^{c\lceil 1/\eps\rceil}$
 for some constant $c\in\bbN$.
 From Theorem \ref{broise_improved} we conclude that 
$w'w''^*\in W^{32\cdot c\lceil 1/\eps\rceil}.$

\vspace{0.2cm}

Summarizing the above steps, we have shown that there exists a constant $C\in\bbN$, which is independent of $\delta$ and $\eps$, such that $W^{C\lceil 1/\eps\rceil}$ contains a neighborhood of the identity in the strong operator topology. This shows that $\cP\cU(\cM)$ has the invariant Steinhaus property.
\end{proof}

Actually the proof of Proposition \ref{prop_nbhd} will allow us to conclude the uniqueness of the Polish group topology of $\cP\cU(\cM)$.
However, we first want to conclude the main theorem in this section from Proposition \ref{prop_auto_Steinhaus} and Proposition \ref{prop_nbhd}.

\begin{thm}
 The projective unitary group of a separable $\mathrm{II_1}$-factor, endowed with the strong operator topology, has the invariant automatic continuity property.
\label{thm_auto_cont}
\end{thm}


Our strategy to obtain Theorem \ref{thm_auto_cont} mainly used the existence of elements of a certain size in a fixed power of every conjugacy-invariant countably syndetic set and our bounded normal generation results. We hope that this strategy leads to more new examples of groups having the invariant automatic continuity property.
%

%
%


As an easy application of Theorems \ref{thm_auto_PUn} and \ref{thm_auto_cont} we see that $\cP\cU(\cM)$ and $\cP\cU(\cM)$ have a unique Polish SIN group topology. 
This is of course valid for any separable topological group with the invariant automatic continuity property. In particular, $\cP\cU(n)$ and $\cSU(n)$ carry a unique Polish SIN group topology. 
In the case $\cP\cU(n)$ it is already known that it also has a unique Polish group topology, see \cite[Theorem 11]{GP-08}. We now want to extend this to II${}_1$-factors.


Now we will make use of the proof of Proposition \ref{prop_nbhd} to conclude the uniqueness of the Polish group topology on $\cP\cU(\cM)$ for any separable $\mathrm{II_1}$-factor. For this purpose, we need the work of Gartside-Peji\'c  \cite[Theorem 8]{GP-08}. 
We first need to explain some notions. 
A \textbf{verbal set}\index{verbal set} is a subset of $G$ of the form
$\set{w(g_1,\ldots,g_n;u_1,\ldots,u_m)\mid g_1,\ldots,g_n\in G},$
where $w$ is a free word and $u_1,\ldots,u_m\in G$.
Verbal sets are forward images under the maps $w$.
For us the most important example of a verbal set is the conjugacy class $\set{gug^{-1}\mid g\in G}$ of an element $u\in G$ -- or a product of conjugacy classes.

We say that a collection $\cN$ of subsets of a topological space $X$ is a \textbf{network}\index{network} if for every $x\in V$ with $V$ open in $X$, there exists $N\in\cN$ such that $x\in N\subseteq V$.
We can now state \cite[Theorem 8]{GP-08}.

\begin{thm}[Gartside-Peji\'c]\label{thm_Pejic}
 Every Polish group that has a countable network of sets from the $\sigma$-algebra generated by verbal sets has a unique Polish group topology.
\end{thm}

The proof of the following main result is based on Proposition \ref{prop_nbhd}.

\begin{thm}\label{thm_unique_top}
 The projective unitary group $G$ of a separable $\mathrm{II_1}$-factor has a unique Polish group topology. 
\end{thm}
\begin{proof}
 We construct a countable network for $G$. For $n\in\bbN$ we let  $\eps_n\coloneqq 1/n$ and $\delta=\delta(n)<\frac{1}{C\lceil 1/\eps_n^2\rceil}=\frac{1}{Cn^2}$, where $C\in\bbN$ is the universal constant coming from the proof of Proposition \ref{prop_nbhd}. Now choose $u,v\in G$ (only dependent on $n$) such that $\norm{1-u}_2<\delta,\norm{1-v}_2<\delta$ and 
 $$\ell_t(u)\geq\eps_n\tn{ for all }t\in[0,\delta],\qquad \ell_t(v)\geq\delta\tn{ for all }t\in[0,\eps_n].$$ 
 Using the proof of Proposition \ref{prop_nbhd} we conclude the existence of $\delta_0=\delta_0(n)\in(0,\delta)$ (independent of $u$ and $v$) such that
 $$\mathrm{B}_{\delta_0}^{\norm{\cdot}_2}(1)\subseteq N_{\eps_n}\coloneqq \left(u^G\cup u^{-G}\cup v^G\cup v^{-G}\right)^{C\lceil 1/\eps_n\rceil}.$$ 
 However, we have $N_{\eps_n}\subseteq\mathrm{B}_{1/n}^{\norm{\cdot}_2}(1)$, since for every $g\in N_{\eps_n}$  we have
$\norm{1-g}_2 \leq C\lceil 1/\eps_n\rceil \delta \leq 1/n.$
 Fix a countable dense subset $D\subseteq G$. We claim that 
 $\cN\coloneqq \set{gN_{\eps_n} \mid g\in D, n\in\bbN}$
 forms a countable network for $G$. First of all, $\cN$ is contained in  the $\sigma$-algebra generated from verbal sets.
 
 Now let $w\in V$ with $V\subseteq G$ open. Since $V$ is open, we can find $\eps>0$ such that $\mathrm{B}_{\eps}^{\norm{\cdot}_2}(w)\subseteq V$. 
 Let $n\in\bbN$ such that $\eps_n=1/n<\eps/2$. By denseness of $D$ we can choose $v_0\in D$ such that $\norm{v_0-w}_2\leq \delta_0$. Then we have (note that $\delta_0=\delta_0(n)<\eps_n<\eps/2$):
 \begin{align*}
  w\in v_0\mathrm{B}_{\delta_0}^{\norm{\cdot}_2}(1) \subseteq v_0N_{\eps_n}
  \subseteq v_0\mathrm{B}_{1/n}^{\norm{\cdot}_2}(1)
  \subseteq v_0\mathrm{B}_{\eps/2}^{\norm{\cdot}_2}(1)
   \subseteq w\mathrm{B}_{\delta_0}^{\norm{\cdot}_2}(1)\mathrm{B}_{\eps/2}^{\norm{\cdot}_2}(1)
  \subseteq w\mathrm{B}_{\eps}^{\norm{\cdot}_2}(1) \\
  =\mathrm{B}_{\eps}^{\norm{\cdot}_2}(w)
  \subseteq V.
 \end{align*}
 That is, for arbitrary $w\in V$, $V$ open in $G$, we find a set $N\in\cN$ such that $w\in N\subseteq V$, i.e., $\cN$ is a network.
 Since $D$ and $\bbN$ are countable, $\cN$ is countable. Now from Theorem \ref{thm_Pejic} we conclude that $G$ has a unique Polish group topology.
\end{proof}

As a consequence of Theorem \ref{thm_unique_top} we obtain the following further automatic continuity results, which are equivalent to the uniqueness of the Polish group topology by \cite[Lemma 10, Lemma 13]{Pe-07}.

\begin{cor}
 Let $\cM$ denote a separable $\mathrm{II_1}$-factor and let $G$ be its projective unitary group. 
 \begin{enumerate}
  \item[(i)] Every isomorphism from $G$ to a Polish group is continuous.
  \item[(ii)] Every epimorphism from a Polish group to $G$ with closed kernel is continuous.
 \end{enumerate}
\end{cor}

It remains to be an interesting open question to decide if the (projective) unitary group of a II${}_1$-factor has the automatic continuity property or not.

\section*{Acknowledgments}

A.T.\ was supported by ERC StG 277728 and P.A.D. was partially supported by ERC CoG 614195. P.A.D.\ wants to thank Universit\"at Leipzig, the IMPRS Leipzig and the MPI-MIS Leipzig for support and an stimulating research environment. Most of the material in this article is part of the PhD-thesis of the first author.

\bibliographystyle{alpha}

\end{document}